\documentclass[11pt]{article}

\usepackage{amssymb,amsthm,amsmath}
\usepackage[a4paper,left=2.5cm,textwidth=16cm]{geometry}

\def\E{\mathbb E}
\def\p{\mathbb P}

\newcommand*{\ind}[1]{\mathbf{1}_{\{#1\}}}

\def\tX{\tilde{X}}
\newcommand{\R}{\mathbb{R}}
\newcommand{\N}{\mathbb{N}}

\newtheorem{lemma}{Lemma}

\newtheorem{thm}[lemma]{Theorem}
\newtheorem{prop}[lemma]{Proposition}
\newtheorem{cor}[lemma]{Corollary}
\newtheorem{defi}[lemma]{Definition}

\author{Rados{\l}aw Adamczak\thanks{Research partially supported by  MNiSW Grant N N201 608740 and the Foundation for Polish Science}\; and Witold Bednorz\thanks{Research partially supported by  MNiSW Grant N N201 608740}}
\title{Orlicz integrability of additive functionals of Harris
ergodic Markov chains}

\begin{document}
\maketitle
\begin{abstract}
For a Harris ergodic Markov chain $(X_n)_{n\ge 0}$, on a general
state space, started from the so called small measure or from the
stationary distribution  we provide optimal estimates for Orlicz
norms of sums $\sum_{i=0}^\tau f(X_i)$, where $\tau$ is the first
regeneration time of the chain. The estimates are expressed in
terms of other Orlicz norms of the function $f$ (wrt the
stationary distribution) and the regeneration time $\tau$ (wrt the
small measure). We provide applications to tail estimates for
additive functionals of the chain $(X_n)$ generated by unbounded
functions as well as to classical limit theorems (CLT, LIL,
Berry-Esseen).


\end{abstract}

\section{Introduction and notation}
Consider a Polish space $\mathcal{X}$ with the Borel
$\sigma$-field $\mathcal{B}$ and let $(X_n)_{n\ge 0}$ be a time
homogeneous Markov chain on $\mathcal{X}$ with a transition
function $P\colon \mathcal{X}\times \mathcal{B} \to [0,1]$.
Throughout the article we will assume that the chain is Harris
ergodic, i.e. that there exists a unique probability measure $\pi$
on $(\mathcal{X},\mathcal{B})$ such that
\begin{displaymath}
\|P^n(x,\cdot) - \pi\|_{TV} \to 0
\end{displaymath}
for all $x \in \mathcal{X}$, where $\|\cdot\|_{TV}$ denotes the total
variation norm, i.e. $\|\mu\|_{TV} = \sup_{A \in \mathcal{B}}
|\mu(A)|$ for any signed measure $\mu$.

One of the best known and most efficient tools of studying such
chains is the so called regeneration technique
\cite{NummSplit,ANsplit}, which we briefly recall bellow. We
refer the reader to the monographs \cite{Numm}, \cite{MT} and
\cite{Chen} for extensive description of this method and restrict
ourselves to the basics which we will need to formulate and prove
our results.

Below we assume that the chain is Harris ergodic.

One can show that under the above assumptions there exists a set (usually called \emph{small set})
$C \in \mathcal{E}^+ = \{A\in
\mathcal{B}\colon\pi(A)> 0\}$, a positive integer $m$,
$\delta
> 0$ and a Borell probability measure $\nu$ on $\mathcal{X}$
(\emph{small measure}) such that
\begin{align}\label{small_set}
P^m(x,\cdot) \ge \delta \nu(\cdot)
\end{align}
for all $x \in C$. Moreover one can always choose $m$ and $\nu$ in
such a way that $\nu(C) > 0$.

 Existence of the above objects allows for redefining the
chain (possibly on an enlarged probability space) together with
auxiliary regeneration structure. More precisely, one defines the
sequence $(\tX_n)_{n\ge0}$ and a sequence $(Y_n)_{n\ge 0}$ by
requiring that $\tX_0$ have the same distribution as $X_0$ and
specifying the conditional probabilities
\begin{align*}
&\p(Y_k =1,\tX_{km+1} \in dx_1,\ldots,\tX_{(k+1)m-1}\in
dx_{m-1},\tX_{(k+1)m} \in
dy|\mathcal{F}_{km}^{\tX},\mathcal{F}^Y_{k-1},\tX_{km}=x)\\
& = P(Y_k =1,\tX_{km+1} \in dx_1,\ldots,\tX_{(k+1)m-1}\in
dx_{m-1},\tX_{(k+1)m} \in dy|\tX_0 =x)\\
&=\ind{x\in C}\frac{\delta\nu(dy)}{P^m(x,dy)}P(x,dx_1)\cdots
P(x_{m-1},dy),
\end{align*}
where $\mathcal{F}_{km}^{\tX} = \sigma((\tX_{i})_{i\le km})$ and
$\mathcal{F}^Y_{k-1}=\sigma((Y_i)_{i\le k-1})$.

One can easily check that $(\tX_n)$ has the same distribution as
$(X_n)$ and so we may and will identify the two sequences (we will
suppress the tilde). The auxiliary variables $Y_n$ can be used to
introduce some independence which allows to recover many results
for Markov chains from corresponding statements for the
independent (or one-dependent) case. Indeed, observe that if we
define the stopping times
\begin{displaymath}
\tau(0) = \inf\{k\ge 0,Y_k = 1\}, \; \tau(i) =
\inf\{k>\tau(i-1)\colon Y_k=1\},\, i = 1,2,\ldots,
\end{displaymath}
then the blocks $R_{0} =
(X_0,\ldots,X_{\tau(0)m+m-1})$, $R_i = (X_{m(\tau(i-1) +
1)},\ldots,X_{m\tau(i)+m-1})$ are one-dependent, i.e. for all $k$
$\sigma(R_i, i < k)$ is independent of $\sigma(R_i, i> k)$. In the
special case, when $m=1$ (the so called \emph{strongly aperiodic
case}) the blocks $R_i$ are independent. Moreover, for $i \ge 1$ the
blocks $R_i$ form a stationary sequence.

In particular for any function $f \colon \mathcal{X} \to \R$, the
corresponding additive functional $\sum_{i=1}^n f(X_i)$ can be
split (modulo the initial and final segment) into a sum (of random length) of
one-dependent (independent for $m=1$) identically distributed
summands
\begin{displaymath}
s_i(f) = \sum_{j=m(\tau(i)+1)}^{m\tau(i+1)+m-1} f(X_j).
\end{displaymath}

A crucial and very useful fact is the following equality, which
follows from Pitman's occupation measure formula
(\cite{Pit1,Pit2}, see also Theorem 10.0.1 in \cite{MT}).

\begin{align}\label{Pitman_formula}
\E_{\nu} \sum_{i=0}^{\tau(0)} F(X_{mi},Y_i) =
\delta^{-1}\pi(C)^{-1} \E_\pi F(X_0,Y_0),
\end{align}
where by $\E_\mu$ we denote the expectation for the process with
$X_0$ distributed according to the measure $\mu$.

It is also worth noting that the distribution of $s_i(f)$ is equal
to the distribution of
\begin{displaymath}
S = S(f)= \sum_{i=0}^{\tau(0) m + m -1} f(X_i)
\end{displaymath}
provided that $X_0$ is distributed according to $\nu$.

In particular, by (\ref{Pitman_formula}) this easily implies that
\begin{align}\label{mean}
\E s_i(f) = \delta^{-1}\pi(C)^{-1}m\int_\mathcal{X} fd\pi.
\end{align}

The above technique of decomposing additive functionals of Markov
chains into independent or almost independent summands has proven
to be very useful in studying limit theorems for Markov chains
(see e.g. \cite{Numm,MT,Chen,Bo1,Bo2,Jo,RobRos}) as well as in
obtaining non-asymptotic concentration inequalities (see e.g.
\cite{Cle,DFMS,AdMarkovTail,AdBed}). The basic difficulty of this
approach is providing proper integrability for the variable $S$.
This is usually achieved either via pointwise drift conditions
(e.g. \cite{MT,Bax, DFMS, AdBed}), especially important in Markov
Chain Monte Carlo algorithms or other statistical applications,
when not much information regarding the behaviour of $f$ with
respect to the stationary measure is available. Such drift
conditions are also useful for quantifying the ergodicity of the
chain, measured in terms of integrability of the regeneration time
$T = \tau(1) - \tau(0)$ (which via coupling constructions can be translated in
the language of total variation norms or mixing coefficients).

Another line of research is more theoretic and concerns the
behaviour of the stationary chain. It is then natural to impose
conditions concerning integrability of $f$ with respect to the
measure $\pi$ and to assume some order of ergodicity of the chain.

Classical assumptions about integrability of $T$ are of the
form $\E T^\alpha < \infty$ or $\E\exp(\theta T) <
\infty$, which corresponds to polynomial or geometric ergodicity
of the chain. However recently new modified drift conditions have
been introduced \cite{DGM,DFMS}, which give other orders of
integrability of $T$ corresponding to various subgeometric
rates of ergodicity. Chains satisfying such drift conditions
appear naturally in Markov Chain Monte Carlo algorithms or
analysis of nonlinear autoregressive models \cite{DGM}.

From this point of view it is natural to ask questions concerning
more general notions of integrability of the variable $S$. In
this note we will focus on Orlicz integrability. Recall that
$\varphi\colon [0,\infty) \to \R_+$ is called a Young functions if
it is strictly increasing, convex and $\varphi(0) = 0$. For a real
random variable $X$ we define the Orlicz norm corresponding to
$\varphi$ as
\begin{displaymath}
\|X\|_\varphi =\inf\{C > 0\colon \E \varphi(|X|/C) \le 1\}.
\end{displaymath}
The Orlicz space associated to $\varphi$ is the set $L_\varphi$ of
random variables $X$ such that $\|X\|_\varphi < \infty$.

In what follows, we will deal with various underlying measures on
the state space $\mathcal{X}$ or on the space of trajectories of
the chain. To stress the dependence of the Orlicz norm on the
initial distribution $\mu$ of the chain $(X_n)$
we will denote it
by $\|\cdot\|_{\mu,\varphi}$, e.g. $\|S\|_{\pi,\varphi}$ will
denote the $\varphi$-Orlicz norm of the functional $S$ for the
stationary chain, whereas $\|S\|_{\nu,\varphi}$ the
$\varphi$-Orlicz norm of the same functional for the chain started
from initial distribution $\nu$. We will also denote by
$\|f\|_{\mu,\rho}$ the $\rho$-Orlicz norm of the function
$f\colon\mathcal{X} \to \R$ when the underlying probability
measure is $\mu$. Although the notation is the same for Orlicz norms of functionals of the Markov chains and functions on $\mathcal{X}$, the meaning will always be clear from the context and thus should not lead to misunderstanding.

\paragraph{Remarks} 1. Note that the distribution of $T$ is
independent of the initial distribution of the chain and is equal
to the distribution of $\tau(0)+1$ for the chain starting from the
measure $\nu$. Thus $\|T\|_\psi = \|\tau(0)+1\|_{\nu,\psi}$.

2. In \cite{NT2}, the authors consider {\emph ergodicity of order $\psi$} of a Markov chain, for a special class of nondecreasing functions $\psi\colon \N \to \R_+$. They call a Markov chain
ergodic of order $\psi$ iff $\E_\nu\psi^\circ(T) < \infty$, where $\psi^\circ (n) = \sum_{i=1}^n \psi(i)$. Since $\psi^\circ$ can be extended to a convex increasing function, one can easily see that this notion is closely related to the finiteness of a proper Orlicz norm of $T$ (related to properly shifted function $\psi^\circ$).

\medskip

We will be interested in the following two closely related
questions

\paragraph{Question 1}
Given two Young functions $\varphi$ and $\psi$ and a Markov chain
$(X_n)$ such that $\|T\|_\psi < \infty$, what do we have to
assume about $f \colon \mathcal{X} \to \R$ to guarantee that
$\|S\|_{\nu,\varphi} < \infty$ (resp. $\|S\|_{\pi,\varphi} <
\infty$)?

\paragraph{Question 2} Given two Young functions $\rho$ and $\psi$, a Markov chain
$(X_n)$ such that $\|T\|_\psi < \infty$ and $f \colon
\mathcal{X} \to \R$, such that $\|f\|_{\pi,\rho} <
\infty$, what can we say about the integrability of $S$ for the
chain started from $\nu$ or from $\pi$?

\medskip

As it turns out, the answers to both questions are surprisingly
explicit and elementary. We present them in Section
\ref{Section_main} (Theorems \ref{thm_nu}, \ref{thm_pi},
Corollaries \ref{question_2_cor_nu}, \ref{question_2_cor_pi}). The
upper estimates have very short proofs, which rely only on
elementary properties of Orlicz functions and the formula
(\ref{mean}). They are also optimal as can be seen from
Propositions \ref{weak_opt_nu}, \ref{weak_opt_pi} and Theorem
\ref{stron_opt_nu} proven in Section \ref{Section_proofs} by
constructing a general class of examples.

We would like to stress that despite being elementary, both the
estimates and the counterexamples have non-trivial applications
(some of which we present in the last section) and therefore are
of considerable interest. For example when specialized to
$\varphi(x) = x^2$, the estimates give optimal conditions for the
CLT or LIL for Markov chains under assumptions concerning the rate
of ergodicity and integrability of the test functions in the
stationary case.

In the following sections of the article we present the estimates,
demonstrate their optimality and provide applications to limit
theorems and tail estimates. For the reader's convenience we
gather all the basic facts about Orlicz spaces which are used in
the course of the proof in the appendix (we refer the reader to
the monographs \cite{Kr,MalBook,RaoRen} for more detailed account
on this class of Banach spaces).

\section{Main estimates \label{Section_main}}
To simplify the notation in what follows we will write $\tau$
instead of $\tau(0)$.

\subsection{The chain started from $\nu$\label{section_nu_thm}}

\paragraph{Assumption $(A)$} We will assume that
\begin{displaymath}
\lim_{x\to 0}\psi(x)/x = 0\;\textrm{and}\; \psi(1) \ge 1.
\end{displaymath}
Since any Young function on a probability space is equivalent to a
function satisfying this condition (see the definition of
domination and equivalence of functions below) it will not
decrease the generality of our estimates while allowing to
describe them in a more concise manner. In particular it assures
the correctness of the following definition (where by a
generalized Young function we mean a nondecreasing convex function
$\rho\colon [0,\infty)\to [0,\infty]$ with $\rho(0) = 0$,
$\lim_{x\to \infty}\rho(x) = \infty$).

\begin{defi} Let $\varphi$ and $\psi$ be Young functions. Assume that $\psi$ satisfies the assumption $(A)$.
Define the generalized Young function $\rho = \rho_{\varphi,\psi}$ by the
formula
\begin{displaymath}
\rho(x)=\sup_{y \ge 0} \frac{\varphi(xy) - \psi(y)}{y}.
\end{displaymath}
\end{defi}

\begin{thm}\label{thm_nu}  Let $\varphi$ and $\psi$ be Young functions. Assume that $\psi$ satisfies the assumption $(A)$.
Let $\rho = \rho_{\varphi,\psi}$. Then for any Harris ergodic
Markov chain $(X_n)$, a small set $C$ and a measure $\nu$
satisfying (\ref{small_set}), we have
\begin{align}
\Big\|\sum_{j=0}^{m\tau +m-1} f(X_j)\Big\|_{\nu,\varphi} \le
2m\|\tau+1\|_{\nu,\psi}\|f\|_{\pi, \rho}.
\end{align}
\end{thm}

\begin{proof}

Let $a = \|\tau+1\|_{\nu,\psi}$, $b = \|f\|_{\pi,\rho}$. We have
\begin{align*}
\E_\nu\varphi\Big(\frac{S}{ab m}\Big) & =
\E_\nu\varphi\Big(\frac{\sum_{j=0}^{\tau m + m-1} f(X_j)}{ab m
}\Big)\\
&\le \E_\nu \sum_{j=0}^{\tau m + m-1}
\frac{\varphi(f(X_j)b^{-1}(\tau+1)a^{-1})}{(\tau+1)m}\\
&\le \E_\nu \sum_{j=0}^{\tau m + m-1}
\frac{\rho(f(X_j)b^{-1})}{am} + \E_\nu \sum_{j=0}^{\tau m + m-1}
\frac{\psi((\tau+1)
a^{-1})}{(\tau+1)m}\\
&=\delta^{-1}\pi(C)^{-1}  a^{-1}\E_\pi \rho(f(X_0)b^{-1}) +
\E_\nu\psi((\tau + 1)a^{-1}),
\end{align*}
where the first inequality follows from Jensen's inequality, the
second one from the definition of the function $\rho$ and the last
equality from (\ref{mean}). Let us now notice that another
application of (\ref{mean}) gives
\begin{displaymath}
\E_\nu (\tau+1) = \delta^{-1}\pi(C)^{-1}.
\end{displaymath}
Thanks to the assumption $\psi(1) \ge 1$, we have
$\E_\nu\psi((\tau+1)\delta\pi(C)) \ge
\psi(\E_\nu(\tau+1)\delta\pi(C))= \psi(1) \ge 1$, which implies
that $a \ge \delta^{-1}\pi(C)^{-1}$. Combined with the definition
of $a$ and $b$ this gives
\begin{displaymath}
\E_\nu\varphi\Big(\frac{S}{ab m}\Big) \le 2
\end{displaymath}
and hence $\E_\nu \varphi(S/(2abm)) \le \E_\nu
2^{-1}\varphi(S/abm) \le 1$, which ends the proof.
\end{proof}

As one can see the proof is very simple. At the same time, it
turns out that the estimate given in Theorem \ref{thm_nu} is
optimal (up to constants) and thus answers completely Question 1
for the chain starting from $\nu$. Below we present two results on
optimality of Theorem \ref{thm_nu} whose proofs are postponed to
the next section.

\paragraph{Domination and equivalence of functions}
Consider two functions $\rho_1, \rho_2 \colon [0,\infty)\to
[0,\infty]$. As is classical in the theory of Orlicz spaces with
respect to probabilistic measures, we say that $\rho_2$ dominates
$\rho_1$ (denoted by $\rho_1 \preceq \rho_2$) if there exist
positive constants $C_1,C_2$ and $x_0$, such that
\begin{align}\label{equivalence_eq}
\rho_1(x) \le C_1\rho_2(C_2x)
\end{align}
for $x\ge x_0$. One can easily check that if $\rho_i$ are Young
functions then $\rho_1 \preceq \rho_2$ iff there is an inclusion
and comparison of norms between the corresponding Orlicz spaces.
We will say that $\rho_1$ and $\rho_2$ are equivalent ($\rho_1
\simeq \rho_2$) iff $\rho_1 \preceq \rho_2$ and $\rho_2 \preceq
\rho_1$. One can also easily check that two Young functions are
equivalent iff they define equivalent Orlicz norms (and the same remains true for functions equivalent to Young functions). Note also that
if (\ref{equivalence_eq}) holds and $\rho_2$ is a Young function
then $\rho_1(x) \le \rho_2(\max(C_1,1)C_2x)$.

\vskip0.7cm \noindent Our first optimality result is
\begin{prop}[Weak optimality of Theorem
\ref{thm_nu}]\label{weak_opt_nu} Let $\varphi$ and $\psi$ be as in
Theorem \ref{thm_nu}. Assume that a Young function $\rho$ has the
property that for every Harris ergodic chain $(X_n)$, a small set
$C$, a small measure $\nu$ with $\|\tau\|_{\nu,\psi} < \infty$ and
every function $f\colon \mathcal{X} \to \R$ such that
$\|f\|_{\pi,\rho}< \infty$, we have $\|S(f)\|_{\nu,\varphi} <
\infty$. Then $\rho_{\varphi,\psi} \preceq \rho$.
\end{prop}

It turns out that if we assume something more about the functions
$\varphi$ and $\psi$, the above proposition can be considerably
strengthened.

\begin{thm}[Strong optimality of Theorem \ref{thm_nu}] \label{stron_opt_nu} Let $\varphi, \psi$ and $\rho$ be as in
Theorem \ref{thm_nu}. Assume additionally that $\varphi^{-1}\circ
\psi$ is equivalent to a Young function. Let $Y$ be a random
variable such that $\|Y\|_{\rho} = \infty$. Then there exists a
Harris ergodic Markov chain $(X_n)$ on some Polish space
$\mathcal{X}$, with stationary distribution $\pi$, a small set
$C$, a small measure $\nu$ and a function $f\colon \mathcal{X}\to
\R$, such that the distribution of $f$ under $\pi$ is equal to the
law of $Y$, $\|\tau\|_{\nu,\psi} < \infty$  and
$\|S(f)\|_{\nu,\varphi} = \infty$.
\end{thm}

\paragraph{Remarks} 1. In the last section we will see that the above theorem for
$\varphi(x) = x^2$ can be used to construct examples of chains
violating the central limit theorem.

\noindent  2. We do not know if the additional assumption on convexity of
$\varphi^{-1}\circ \psi$ is needed in the above Theorem.

\noindent 3. In fact in the construction we provide the set $C$ is an atom for the chain (i.e. in the minorization condition $m=1$ and $\delta=1$).

\vskip0.7cm

The above results give a fairly complete answer to Question 1 for
a chain started from a small measure. We will now show that
Theorem \ref{thm_nu} can be also used to derive the answer to
Question 2.

Recall that the Legendre transform of a function $\rho\colon
[0,\infty)\to \R_+$ is defined as $\rho^\ast = \sup\{xy- \rho(y)
\colon y \ge 0\}$.
 Our answer to Question
2 is based on the following observation (which will also be used
in the proof of Theorem \ref{stron_opt_nu}).

\begin{prop}\label{legendre_prop}
For any Young functions $\varphi, \psi$ satisfying Assumption (A),
the function $\rho = \rho_{\varphi,\psi}$ is equivalent to
$\eta^\ast$, where $\eta = (\psi^\ast)^{-1}\circ \varphi^\ast$.
More precisely, for any $x \ge 0$,
\begin{equation}\label{legendre_ineq}
2\eta^\ast(2^{-1}x) \le \rho(x) \le 2^{-1}\eta^\ast(2x).
\end{equation}
\end{prop}

Before we prove the proposition let us derive the immediate
corollary, whose optimality will also be shown in the next
section.

\begin{cor}\label{question_2_cor_nu}  Let $\rho$ and $\psi$ be two Young functions. Assume that $\psi$ satisfies the assumption (A). Then for any Harris ergodic Markov chain, small set $C$,
small measure $\nu$ and any $f\colon \mathcal{X} \to \R$ we have
\begin{displaymath}
\|S\|_{\nu,\tilde{\varphi}} \le
4m\|(\tau+1)\|_{\nu,\psi}\|f\|_{\pi,\rho},
\end{displaymath}
where $\tilde{\varphi} = (\psi^\ast \circ\rho^\ast)^\ast$.
\end{cor}

\begin{proof}[Proof of Proposition \ref{legendre_prop}]
Using the fact that $\varphi^{\ast\ast} = \varphi$ we get
\begin{align*}
\rho(x) &= \sup_{y \ge 0} \frac{\varphi(xy) - \psi(y)}{y} =
\sup_{y\ge 0}\sup_{z\ge 0} \frac{xyz - \varphi^\ast(z) -
\psi(y)}{y}\\
&= \sup_{z\ge 0} \Big(xz - \inf_{y \ge 0} \frac{\varphi^\ast(z) +
\psi(y)}{y}\Big) =\tilde{\eta}^\ast(x),
\end{align*}
where $\tilde{\eta}(x) = \inf_{y \ge 0} (\varphi^\ast(z) +
\psi(y))y^{-1}$. Note that as a function of $y$,
$\varphi^\ast(z)y^{-1}$ decreases whereas $\psi(y)y^{-1}$
increases, so for all $z \ge 0$ we have
\begin{displaymath}
\frac{\varphi^\ast(z)}{y_0}\le \eta(z) \le
2\frac{\varphi^\ast(z)}{y_0},
\end{displaymath}
where $y_0$ is defined by the equation $\varphi^\ast(z) =
\psi(y_0)$, i.e. $y_0 = \psi^{-1}(\varphi(z))$. In combination with
Lemma \ref{lemma_inverses} from the Appendix, this yields
\begin{displaymath}
\frac{1}{2}\eta(z) \le \tilde{\eta}(z) \le 2\eta(z),
\end{displaymath}
which easily implies that $2\eta^\ast(x/2) \le \rho(x) \le
2^{-1}\eta^\ast(2x)$ and thus ends the proof.
\end{proof}

We also have the following Proposition whose prove is deferred to Section \ref{Section_proofs}.

\begin{prop}\label{Question_2_opt_nu} Let $\psi$ and $\rho$ be as in Corollary
\ref{question_2_cor_nu} and let $\varphi$ be a Young function such
that for every Markov chain $(X_n)$, small set $C$, small measure
$\nu$  and $f\colon \mathcal{X} \to \R$ with $\|\tau\|_{\nu,\psi}
< \infty$ and $\|f\|_{\pi,\rho} < \infty$ we have
$\|S\|_{\nu,\varphi} < \infty$. Then $\varphi \preceq
(\psi^\ast\circ\rho^\ast)^\ast$.
\end{prop}

\paragraph{Examples}

Let us now take a closer look at consequences of our theorems for
classical Young functions. The following examples are
straightforward and rely only on theorems presented in the last
two sections and elementary formulas for Legendre transforms of
classical Young functions. The formulas we present here will be
used in Section \ref{Section_applications}. We also note that below we consider functions of the form $x\mapsto \exp(x^\alpha)-1$ for $\alpha \in(0,1)$. Formally such functions are not Young functions but it is easy to see that they can be modified for small values of $x$ in such a way that they become Young functions. It is customary to define $\|X\|_{\psi_\alpha} = \inf\{C> 0 \colon \E\exp((|X|/C)^\alpha) \le 2\}$. Under such definition $\|\cdot\|_{\psi_\alpha}$ is a quasi-norm, which can be shown to be equivalent to the Orlicz norm corresponding to the modified function.

\begin{enumerate}
\item If $\varphi(x) =x^p$ and $\psi(x)= x^r$, where $r > p\ge 1$
then $\rho_{\varphi,\psi}(x) \simeq x^{\frac{p(r-1)}{r-p}}$.

\item If $\varphi(x) = \exp(x^\alpha) - 1$ and $\psi(x) =
\exp(x^\beta)-1$, where $\beta \ge \alpha$ then
$\rho_{\varphi,\psi}(x) \simeq
\exp(x^{\frac{\alpha\beta}{\beta-\alpha}})-1$.

\item If $\varphi(x) = x^p$ and $\psi(x) = \exp(x^\beta)-1$, where
$\beta >0$ then $\rho_{\varphi,\psi}(x) \simeq
x^p\log^{(p-1)/\beta}x$.

\item If $\psi(x) = x^r$ and $\rho(x) = x^p$ then $\varphi(x)
\simeq x^{\frac{rp}{r+p-1}}$.

\item If $\psi(x) = \exp(x^\beta)-1$ and $\rho(x) =
\exp(x^\alpha)-1$ ($\alpha,\beta > 0$), then $\varphi(x) \simeq
\exp(x^{\frac{\alpha\beta}{\alpha+\beta}})-1$.

\item If $\psi(x) = \exp(x^\beta)-1$ ($\beta>0$) and $\rho(x) =
 x^p$ ($p\ge1$), then $\varphi(x) \simeq \frac{x^p}{\log^{(p-1)/\beta} x}$.
\end{enumerate}

\subsection{The stationary case}
We will now present answers to questions 1 and 2 in the stationary case. Let us start with the following

\begin{defi} Let $\varphi$ and $\psi$ be Young functions. Assume
that $\lim_{x\to 0}\psi(x)/x = 0$ and define the generalized Young
function $\zeta = \zeta_{\varphi,\psi}$ by the formula
\begin{displaymath}
\zeta(x) = \sup_{y\ge 0} (\varphi(xy) - y^{-1}\psi(y)).
\end{displaymath}
\end{defi}

The function $\zeta$ will play in the stationary case a role analogous to the one of function $\rho$ for the chain started from the small measure.

\begin{thm}\label{thm_pi}Let $\varphi$ and $\psi$ be Young functions, $\lim_{x\to0}\psi(x)/x = 0$.
Let $\zeta = \zeta_{\varphi,\zeta}$. Then for any Harris ergodic
Markov chain $(X_n)$, small set $C$ and small measure $\nu$ we
have
\begin{align}
\Big\|\sum_{j=0}^{m\tau +m-1} f(X_j)\Big\|_{\pi,\varphi} \le
m\|\tau+1\|_{\nu,\psi}\Big(1+
\delta\pi(C)\|\tau+1\|_{\nu,\psi}\Big)\|f\|_{\pi, \zeta}.
\end{align}
\end{thm}

\begin{proof} The proof is very similar to the proof of Theorem
\ref{thm_nu}, however it involves one more use of Pitman's formula
to pass from the stationary case to the case of the chain started
from $\nu$.

Consider any functional $F\colon \mathcal{X}^\N \times \{0,1\}^\N
\to\R$ (measurable wrt the product $\sigma$-field) on the space of
the trajectories of the process $(X_n,Y_n)_{n\ge 0}$ (recall from
the introduction that we identify $X_n$ and $\tilde{X}_n$). By the
definition of the split chain we have for any $i\in \N$,
\begin{displaymath}
\E (F((X_j)_{j\ge im},(Y_j)_{j\ge
i})|\mathcal{F}_{im}^X,\mathcal{F}_{i}^Y) = G(X_{im},Y_i),
\end{displaymath}
where $G(x,y) = \E_{(x,y)} F((X_i)_{i\ge 0}, (Y_i)_{i\ge 0}) = \E
F((X_i)_{i\ge 0}, (Y_i)_{i\ge 0}|X_0=x,Y_0=y)$. In particular for
the functional
\begin{displaymath}
F((X_i)_{i\ge 0},(Y_i)_{i\ge 0}) =
\varphi((abm)^{-1}\sum_{i=0}^{m\tau+m-1}f(X_i)),
\end{displaymath}
where $a  = \|\tau+1\|_{\nu,\psi}$ and $b = \|f\|_{\pi,\zeta}$, we
have
\begin{align*}
&\E_\pi \varphi((abm)^{-1}\sum_{i=0}^{m\tau+m-1}f(X_{i})) =
\E_\pi G(X_0,Y_0) = \delta \pi(C) \E_\nu \sum_{i=0}^{\tau} G(X_{im},Y_i) \\
 & = \delta \pi(C)\sum_{i=0}^\infty\E_\nu
 G(X_{im},Y_i)\ind{i\le \tau}
= \delta \pi(C)\sum_{i=0}^\infty \E_\nu \E\Big(F((X_j)_{j\ge
im},(Y_j)_{j\ge
i})|\mathcal{F}^X_{im},\mathcal{F}_{i}^Y\Big)\ind{i\le
\tau}\\
&=\delta \pi(C)\sum_{i=0}^\infty \E_\nu
\varphi\Big((abm)^{-1}\sum_{j=im}^{m\tau+m-1} f(X_j)\Big)\ind{i\le
\tau}
= \delta \pi(C)\E_\nu \sum_{i=0}^{\tau}
\varphi\Big((abm)^{-1}\sum_{j=im}^{m\tau+m-1} f(X_j)\Big)\\
&\le \delta\pi(C)\E_\nu \sum_{i=0}^\tau
\sum_{j=im}^{m\tau+m-1}\frac{1}{m(\tau-i+1)}
\varphi((ab)^{-1}(\tau - i+1)f(X_j)) \\
& = \delta\pi(C)\E_\nu \sum_{j=0}^{m\tau+m-1} \sum_{i=0}^{\lfloor
j/m\rfloor} \frac{1}{m(\tau-i+1)} \varphi((ab)^{-1}(\tau -
i+1)f(X_j))\\
&\le \delta\pi(C)\E_\nu \sum_{j=0}^{m\tau+m-1} \frac{\lfloor
jm^{-1}\rfloor +1}{m(\tau+1)}\varphi((ab)^{-1}(\tau+1)f(X_j)),
\end{align*}
where the second equality follows from (\ref{Pitman_formula}) and
the two last inequalities from the convexity of $\varphi$.

We thus obtain
\begin{align*}
\E_\pi \varphi((abm)^{-1}S(f)) &\le
\delta\pi(C)m^{-1}\E_\nu\sum_{i=0}^{m\tau+m-1}\varphi((ab)^{-1}(\tau+1)
f(X_i))\\
&\le
\delta\pi(C)m^{-1}\E_\nu\sum_{i=0}^{m\tau+m-1}\zeta(b^{-1}f(X_i))
+
\delta\pi(C)a\E_\nu \psi(a^{-1}(\tau+1)) \\
& \le  \E_\pi \zeta(b^{-1}f(X_0)) + \delta\pi(C)a\E_\nu
\psi(a^{-1}(\tau+1)) \le 1+\delta\pi(C)a,
\end{align*}
which ends the proof.
\end{proof}

\paragraph{Remark} The dependence of the estimates presented in the above theorem on
 $\|\tau+1\|_{\nu,\psi}$ cannot be improved in the case of general
 Orlicz functions, since for $\varphi(x) = x$, $\psi(x) = x^2$,
and $f \equiv 1$ we have $\|S(f)\|_{\pi,\varphi} =
\E_{\pi}(\tau+1) \simeq \E_\nu (\tau+1)^2 =
\|\tau+1\|_{\nu,\psi}^2$. However under additional assumptions on
the growth of $\varphi$ one can obtain a better estimate and
replace the factor $1 + \delta\pi(C)\|\tau+1\|_{\nu,\psi}$
by $g(1+\delta\pi(C)\|\tau+1\|_{\nu,\psi})$, where $g(r) =
\sup_{x > 0} x/\varphi^{-1}(\varphi(x)/r)$. For rapidly growing
$\varphi$ and large $\|\tau+1\|_{\nu,\psi}$ this may be an
important improvement. It is also elementary to check that for
$\phi(x) = \exp(x^\alpha) - 1$, we can use $g(r) \simeq
\log^{1/\alpha}(r)$.

\paragraph{} Just as in the case of Theorem \ref{thm_nu}, the estimates
given in Theorem \ref{thm_pi} are optimal. Below we state the
corresponding optimality results, deferring their proofs to
Section \ref{Section_proofs}.

\begin{prop}[Weak optimality of Theorem
\ref{thm_pi}]\label{weak_opt_pi} Let $\varphi$ and $\psi$ be as in
Theorem \ref{thm_pi}. Assume that a Young function $\zeta$ has the
property that for every Harris ergodic chain $(X_n)$, small set
$C$ and small measure $\nu$ with $\|\tau\|_{\nu,\psi} < \infty$
and every function $f\colon \mathcal{X} \to \R$ such that
$\|f\|_{\pi,\zeta}< \infty$, we have $\|S(f)\|_{\pi,\varphi} <
\infty$. Then $\zeta_{\varphi,\psi} \preceq \zeta$.
\end{prop}

\begin{thm}[Strong optimality of Theorem \ref{thm_pi}] \label{stron_opt_pi} Let $\varphi, \psi$ and $\zeta$ be as in
Theorem \ref{thm_pi}. Let $\tilde{\psi}(x) = \psi(x)/x$ and assume
additionally that the function $\eta = \varphi^{-1}\circ
\tilde{\psi}$ is equivalent to a Young function. Let $Y$ be a
random variable such that $\|Y\|_{\zeta} = \infty$. Then there
exists a Harris ergodic Markov chain $(X_n)$ on some Polish space
$\mathcal{X}$ with stationary distribution $\pi$, small set $C$,
small measure $\nu$ and a function $f\colon \mathcal{X}\to \R$,
such that the distribution of $f$ under $\pi$ is equal to the law
of $Y$, $\|\tau\|_{\nu,\varphi} < \infty$ and
$\|S(f)\|_{\pi,\varphi} = \infty$.
\end{thm}

\begin{prop}\label{legendre_prop_pi}
For any Young functions $\varphi, \psi$ such that $\lim_{x\to
\infty}\psi(x)/x = 0$, the function $\zeta = \zeta_{\varphi,\psi}$
is equivalent to $\varphi\circ\eta^\ast$, where $\eta(x) =
\varphi^{-1}(\psi(x)/x)$. More precisely, for all $x \ge 0$,
\begin{displaymath}
\varphi(\eta^\ast(x))\le \zeta(x) \le
\frac{1}{2}\varphi(\eta^\ast(2x)).
\end{displaymath}
\end{prop}

\begin{proof}
Thanks to the assumption on $\psi$, we have $\lim_{y\to 0}
\varphi(xy) - \psi(y)/y = 0$, so we can restrict our attention to
$y
>0$, such that $\varphi(xy)
> \psi(y)/y$ (note that if there are no such $y$, then $\eta^\ast(x)=\zeta(x)=0$ and the inequalities of the proposition are trivially true). For such $y$, by convexity of $\varphi$ we obtain
\begin{displaymath}
\varphi(xy) - \psi(y)/y \ge \varphi(xy - \varphi^{-1}(\psi(y)/y))
\end{displaymath}
and
\begin{displaymath}
\varphi(xy) - \psi(y)/y \le \varphi(xy) - \psi(y)/(2y) \le
\frac{1}{2}\varphi(2xy - \varphi^{-1}(\psi(y)/y)),
\end{displaymath}
which, by taking the supremum over $y$, proves the proposition.
\end{proof}

\begin{lemma}\label{auxiliary_lemma}
Assume that $\zeta$ and $\psi$ are Young functions,
$\tilde{\psi}(x) =\psi(x)/x$ is strictly increasing and
$\tilde{\psi}(0) = 0$, $\tilde{\psi}(\infty) = \infty$. Let the
function $\kappa$  be defined by
\begin{align}\label{inverse_definition}
\kappa^{-1}(x) = \zeta^{-1}(x)\tilde{\psi}^{-1}(x)
\end{align}
for all $x \ge 0$. Then there exist constants $K,x_0 \in
(0,\infty)$ such that for all $x \ge x_0$,
\begin{align}\label{auxiliary_est}
K^{-1}x \le (\vartheta^\ast)^{-1}(x)\tilde{\psi}^{-1}(\kappa(x))
\le 2x
\end{align}
where $\vartheta = \kappa^{-1}\circ \tilde{\psi}$.

Moreover the function $\tilde{\zeta} = \kappa\circ \vartheta^\ast$
is equivalent to $\zeta$.
\end{lemma}

\begin{proof}
Note first that $\vartheta(x) = \zeta^{-1}(\tilde{\psi}(x))x$, and
so $\vartheta$ is equivalent to a Young function (e.g. by Lemma
\ref{asymptotic_convexity_MS} in the Appendix). The inequalities
(\ref{auxiliary_est}) follow now by Lemma \ref{lemma_inverses}
from the Appendix.

Moreover
\begin{displaymath}
\tilde{\zeta}^{-1}(x) = (\vartheta^{\ast})^{-1}(\kappa^{-1}(x))
\end{displaymath}
and thus by (\ref{auxiliary_est}) for $x$ sufficiently large,
\begin{displaymath}
K^{-1}\zeta^{-1}(x) =
K^{-1}\frac{\kappa^{-1}(x)}{\tilde{\psi}^{-1}(x)}\le
\tilde{\zeta}^{-1}(x)\le
2\frac{\kappa^{-1}(x)}{\tilde{\psi}^{-1}(x)} = 2\zeta^{-1}(x),
\end{displaymath}
which clearly implies that $\tilde{\zeta}(K^{-1}x) \le \zeta(x)\le
\tilde{\zeta}(2x)$ for $x$ large enough.
\end{proof}

If now $\varphi$ is a Young function such that $\varphi\preceq
\kappa$, then $\varphi((\varphi^{-1}\circ\tilde{\psi})^\ast)
\preceq \kappa((\kappa^{-1}\circ\tilde{\psi})^\ast) \simeq \zeta$.
Thus the above Lemma, together with Theorem \ref{thm_pi} and
Proposition \ref{legendre_prop_pi} immediately gives the following

\begin{cor}\label{question_2_cor_pi} Assume that $\zeta$ and
$\psi$ are Young functions, $\tilde{\psi}(x) =\psi(x)/x$ is
strictly increasing, $\tilde{\psi}(0) = 0$, $\tilde{\psi}(\infty)
= \infty$. Let the function $\kappa$ be defined by
(\ref{inverse_definition}). If $\varphi$ is a Young function such
that $\varphi\preceq \kappa$, then there exists $K < \infty$, such
that for any Harris ergodic Markov chain $(X_n)$ on $\mathcal{X}$,
small set $C$, small measure $\nu$ and $f\colon \mathcal{X} \to
\R$,
\begin{align}\label{question_2_pi_eq}
\|S(f)\|_{\pi,\varphi} \le K\|\tau+1\|_{\nu,\psi}\Big(1+
\delta\pi(C)\|\tau+1\|_{\nu,\psi}\Big)\|f\|_{\pi,\zeta}.
\end{align}
\end{cor}

\paragraph{Remark} For slowly growing functions $\psi$ and $\zeta$
there may be no Orlicz function $\varphi$ such that $\varphi \preceq \kappa$. This is
not surprising since as we will see from the construction
presented in Section \ref{section_counterexample} the
$\pi$-integrability of $S(f)$ is closely related to integrability
of functions from a point-wise product of Orlicz spaces. In
consequence $S(f)$ may not even be integrable.

We have the following optimality result corresponding to Corollary \ref{question_2_cor_pi}. Its proof will be presented in the next section.

\begin{prop}\label{Question_2_opt_pi} Assume that $\zeta$ and
$\psi$ are Young functions, $\tilde{\psi}(x) =\psi(x)/x$ is
strictly increasing, $\tilde{\psi}(0) = 0$, $\tilde{\psi}(\infty)=
\infty$. Let the function $\kappa$ be defined by
(\ref{inverse_definition}) and let $\varphi$ be a Young function
such that for every ergodic Markov chain $(X_n)$, small set $C$,
small measure $\nu$ and $f\colon \mathcal{X} \to \R$ with
$\|\tau\|_{\nu,\psi}< \infty$ and $\|f\|_{\pi,\zeta} < \infty$, we
have $\|S(f)\|_{\pi,\varphi}< \infty$. Then $\varphi \preceq
\kappa$.
\end{prop}

\paragraph{Remark} By convexity of $\varphi$, the condition $\varphi\preceq \kappa$ holds iff there
exists a constant $K < \infty$ and $x_0 >0$, such that
\begin{displaymath}
K\varphi^{-1}(x) \ge \tilde{\psi}^{-1}(x)\zeta^{-1}(x)
\end{displaymath}
for $x > x_0$. Thus under the assumptions that $\tilde{\psi}$ is
strictly increasing $\tilde{\psi}(0) = 0$, $\tilde{\psi}(\infty) =
\infty$, the above condition characterizes the triples of Young
functions such that $\|f\|_{\pi,\zeta}< \infty$ implies $\|S(f)\|_{\pi,\varphi}<\infty$ for all Markov
chains with $\|\tau\|_{\nu,\psi}<\infty$.

\paragraph{Examples}

Just as in the previous section we will now present some concrete
formulas for classical Young functions, some of which will be used in
Section \ref{Section_applications} to derive tail inequalities for
additive functionals of stationary Markov chains.

\begin{enumerate}
\item If $\varphi(x) =x^p$ and $\psi(x)= x^r$, where $r > p+1 \ge
2$ then $\zeta_{\varphi,\psi}(x) \simeq x^{\frac{p(r-1)}{r-p-1}}$.

\item If $\varphi(x) = \exp(x^\alpha) - 1$ and $\psi(x) =
\exp(x^\beta)-1$, where $\beta \ge \alpha$ then
$\zeta_{\varphi,\psi}(x) \simeq
\exp(x^{\frac{\alpha\beta}{\beta-\alpha}})-1$.

\item If $\varphi(x) = x^p$ and $\psi(x) = \exp(x^\beta)-1$, where
$\beta >0$ then $\zeta_{\varphi,\psi}(x) \simeq x^p\log^{p/\beta}
x$.

\item If $\psi(x) = x^r$ and $\zeta(x) = x^p$ ($r\ge 2, p \ge (r-1)/(r-2)$)
then $\varphi(x) \simeq x^{\frac{(r-1)p}{r+p-1}}$.

\item If $\psi(x) = \exp(x^\beta)-1$ and $\zeta(x) =
\exp(x^\alpha)-1$ ($\alpha,\beta > 0$), then $\varphi(x) \simeq
\exp(x^{\frac{\alpha\beta}{\alpha+\beta}})-1$.

\item If $\psi(x) = \exp(x^\beta)-1$ ($\beta>0$) and $\zeta(x) =
 x^p$ ($p>1$), then $\varphi(x) \simeq \frac{x^p}{\log^{p/\beta} x}$.
\end{enumerate}

\section{Proofs of optimality\label{Section_proofs}}

\subsection{Main counterexample\label{section_counterexample}}

 We will now
introduce a general construction of a Markov chain which will
serve as an example in proofs of all our optimality theorems.

Let $\mathcal{S}$ be a Polish space and let $\alpha$ be a Borel
probability measure on $\mathcal{S}$. Consider two functions
$\tilde{f}\colon \mathcal{S} \to \R$ and $h \colon \mathcal{S} \to
\N\setminus\{0\}$. We will construct a Markov chain on some Polish
space $\mathcal{X} \supseteq\mathcal{S}$, a small set $C \subseteq
\mathcal{X}$, a probability measure $\nu$ and a function $f\colon
\mathcal{X} \to \R$, possessing the following properties.

\begin{center}
{\bf Properties of the chain}
\end{center}
\begin{itemize}
\item[(i)] The condition (\ref{small_set}) is satisfied with $m=1$
and $\delta = 1$ (in other words $C$ is an atom for the chain),

\item[(ii)] $\nu(\mathcal{S})= 1$,

\item[(iii)] for any $x \in \mathcal{S}$, $\p_x(\tau + 1= h(x)) =
1$,

\item[(iv)] for any $x \in \mathcal{S}$, $\p_x(S(f) =
\tilde{f}(x)h(x)) = 1$,

\item[(v)] For any function $G\colon \R \to \R$ we have
\begin{displaymath}
\E_\nu G(S(f)) = R \int_{\mathcal{S}} G(\tilde{f}(x)h(x))
h(x)^{-1}\alpha(dx)
\end{displaymath}
and
\begin{displaymath}
\E_\nu G(\tau+1) = R\int_{\mathcal{S}}G(h(x))h(x)^{-1}\alpha(dx),
\end{displaymath}
where $R = (\int_C h(y)^{-1}\alpha(dy))^{-1}$,

\item[(vi)] $(X_n)$ admits a unique stationary distribution $\pi$
and the law of $f$ under $\pi$ is the same as the law of
$\tilde{f}$ under $\alpha$,

\item[(vii)] for any nondecreasing function $F\colon \mathcal{X}
\to \R$,
\begin{displaymath}
\E_\pi F(|S(f)|) \ge \frac{1}{2}\int_{\mathcal{S}}
F(h(x)|\tilde{f}(x)|/2)\alpha(dx).
\end{displaymath}

\item[(viii)] if $\alpha(\{x\colon h(x) = 1\}) >0$ then the chain
is Harris ergodic.
\end{itemize}

\paragraph{Construction of the chain}

Let $\mathcal{X} = \bigcup_{n=1}^\infty \{x\in \mathcal{S} \colon
h(x) \ge n\}\times \{n\}$. As a disjoint union it clearly
possesses a natural structure of a Polish space inherited from
$\mathcal{S}$. Formally $\mathcal{S} \not \subseteq \mathcal{X}$
but it does not pose a problem as we can clearly identify
$\mathcal{S}$ with $\mathcal{S}\times \{1\} =\{x\in
\mathcal{S}\colon h(x) \ge 1\}\times\{1\}$.

The dynamics of the chain will be very simple.
\begin{itemize}
\item If $X_n = (x,i)$ and $h(x) > i$, then with probability one
$X_{n+1} = (x,i+1)$.

\item If $X_n = (x,i)$ and $h(x) = i$, then $X_{n+1} = (y,1)$,
where $y$ is distributed according to the probability measure

\begin{align}\label{nu_definition}
\nu(dy) = Rh(y)^{-1}\alpha(dy).
\end{align}
\end{itemize}

More formally, the transition function of the chain is given by
\begin{displaymath}
P((x,i),A) = \left\{
\begin{array}{ccc}
\delta_{(x,i+1)}(A)&\textrm{if}& i < h(x) \\
\nu(\{y\in\mathcal{S}\colon (y,1) \in A\}) &\textrm{if}&
i = h(x).
\end{array}
 \right.
\end{displaymath}

In other words the chain describes a particle, which after
departing from a point $(x,1) \in \mathcal{S}$ changes its 'level'
by jumping deterministically to points $(x,2), \ldots, (x,h(x))$
and then goes back to 'level' one by selecting the first
coordinate according to the measure $\nu$.

Clearly $\nu(\mathcal{S}) =1$ and so condition (ii) is satisfied.
Note that $\alpha$ and $\nu$ are formally measures on
$\mathcal{S}$, but we may and will sometimes treat them as
measures on $\mathcal{X}$.

Let now $C = \{(x,i) \in \mathcal{X}\colon h(x)= i\}$. Then
$P((x,i),A) = \nu(A)$ for any $(x,i)\in C$ and a Borel subset $A$
of $\mathcal{X}$, which shows that (\ref{small_set}) holds with
$m=1$ and $\delta = 1$ .

Let us now prove condition (iii). Since $C$ is an atom for the
chain, $Y_n = 1$ iff $X_n \in C$. Moreover if $X_0 = (x,1) \simeq
x \in \mathcal{S}$, then $X_{i} = (x,i+1)$ for $i+1 \le h(x)$ and
$\tau = \inf\{i\ge 0\colon X_i \in C\} = \inf\{i\ge 0\colon i+1 =
h(x)\} = h(x) - 1$, which proves property (iii).

To assure that property (iv) holds it is enough to define
\begin{displaymath}
f((x,i)) = \tilde{f}(x),
\end{displaymath}
since then $f(X_0) = (x,1)$ implies that $f(X_n) = \tilde{f}(x)$
for $n \le \tau$.

Condition (v) follows now from properties (ii), (iii) and (iv)
together with formula $(\ref{nu_definition})$.

We will now pass to conditions (vi) and (vii).

By the construction of the chain it is easy to prove that the
chain admits a unique stationary measure $\pi$ given by
\begin{displaymath}
\pi(A\times \{k\}) = \alpha(A)n^{-1}
\end{displaymath}
for $A \subseteq \{x\in \mathcal{S}\colon h(x) = n\}$ and any $k
\le n$ . Thus for any Borel set $B\subseteq \R$ we have
\begin{align*}
\pi(\{(x,i)\in \mathcal{X}\colon f((x,i)) \in B\}) &=
\pi(\{(x,i)\in \mathcal{X}\colon \tilde{f}(x) \in B\}) \\
&= \sum_{n\ge 1} \pi(\{(x,i)\in \mathcal{X} \colon h(x) = n,
\tilde{f}(x)\in B\}\\
& = \sum_{n\ge 1} n\cdot n^{-1}\alpha(\{x \in \mathcal{S}\colon
h(x) = n, \tilde{f}(x)\in B\})\\
& =\alpha(\{x\in \mathcal{S}\colon \tilde{f}(x) \in B\}).
\end{align*}

As for (vii), $X_0= (x,i)$ implies that
\begin{displaymath}
S(f) = (h(x)-i+1)\tilde{f}(x).
\end{displaymath}
Thus, letting $A_{n,k} = \{(x,k)\in \mathcal{X}\colon h(x) = n\}$,
$B_n = \{x\in \mathcal{S}\colon h(x) = n\}$,  we get
\begin{align*}
\E_\pi F(|S(f)|) &= \int_\mathcal{X}F((h(x) - i+1)|\tilde{f}(x)|)
\pi(d(x,i)) = \sum_{n,k}\int_{A_{n,k}} F( (n - k
+1)|\tilde{f}(x)|)\pi(d(x,i))\\
&=\sum_{n}\sum_{k\le n}
\int_{B_n}n^{-1}F((n-k+1)|\tilde{f}(x)|)\alpha(dx) \ge \sum_n
\int_{B_n} \frac{1}{2}F(n|\tilde{f}(x)|/2)\alpha(dx)\\
&=\frac{1}{2}\int_{\mathcal{S}} F(h(x)|\tilde{f}(x)|/2)\alpha(dx),
\end{align*}
proving (vii).

Now we will prove (viii). Note that $A := \{x\in
\mathcal{S}\colon h(x)=1\} \subseteq C$. Thus if $\alpha(A) > 0$ then
also $\nu(C) > 0$, which proves that the chain is strongly
aperiodic (see e.g. chapter 5 of \cite{MT} or Chapter 2 of
\cite{Numm}). Moreover one can easily see that $\pi$ is an
irreducibility measure for the chain and the chain is Harris
recurrent.  Thus by Proposition 6.3. of \cite{Numm} the chain is
Harris ergodic (in fact in \cite{Numm} ergodicity is defined as
aperiodicity together with positiveness and Harris recurrence and
Proposition 6.3. states that this is equivalent to convergence of
$n$-step probabilities for any initial point).

What remains to be proven is condition (i). Since $\pi(C) > 0$ we
have $C\in\mathcal{E}^+$, whereas inequality (\ref{small_set}) for
$m = \delta = 1$ is satisfied by the construction.

\subsection{The chain started from $\nu$}

We will start with the proof of Proposition \ref{weak_opt_nu}. The
chain constructed above will allow us to reduce it to elementary
techniques from the theory of Orlicz spaces.

\begin{proof}[Proof of Proposition \ref{weak_opt_nu}]
Assume that the function $\rho$ does not satisfy the condition
$\rho_{\varphi,\psi} \preceq \rho$. Thus there exists a sequence of
numbers $x_n \to \infty$ such that
\begin{displaymath}
\rho(x_n) < \rho_{\varphi,\psi}(x_n2^{-n}).
\end{displaymath}
By the definition of $\rho_{\varphi,\psi}$ this means that there
exists a sequence $t_n > 0$ such that
\begin{displaymath}
\frac{\varphi(x_nt_n2^{-n})}{t_n} \ge \rho(x_n) +
\frac{\psi(t_n)}{t_n}.
\end{displaymath}
One can assume that $t_n \ge 2$. Indeed, for all $n$ large enough if $t_n \le 2$,
then
\begin{displaymath}
\frac{\varphi((x_n2^{-1})2^{-(n-1)}\cdot 2)}{2} \ge
\frac{\varphi(x_nt_n2^{-n})}{t_n} \ge \rho(x_n) \ge
2\rho(x_n2^{-1}) \ge \rho(x_n2^{-1}) +\frac{\psi(2)}{2}.
\end{displaymath}
Set $\tau_n =\lfloor t_{n}\rfloor$ for $n\ge 1$ and $\tau_0 = 1$.
We have for $n\ge 1$
\begin{align}\label{weak_opt nu_tau}
\frac{\varphi(x_{n}\tau_n2^{1-n})}{\tau_n} \ge\frac{\varphi(x_n
t_{n} 2^{-n})}{t_{n}} \ge \rho(x_n) + \frac{\psi(t_{n})}{t_{n}}
\ge \rho(x_n) + \frac{\psi(\tau_n)}{\tau_n}\ge 1,
\end{align}
where in the last inequality we used assumption (A). Define now
$p_n = C2^{-n}(\psi(\tau_n)/\tau_n + \rho(x_n))^{-1}$, where $C$
is a constant such that $\sum_{n\ge 0} p_n = 1$. Consider a Polish
space $\mathcal{S}$ with a probability measure $\alpha$, a
partition $\mathcal{S} = \bigcup_{n\ge 0} A_n$, $\alpha(A_n) =
p_n$ and two functions $h$ and $\tilde{f}$, such that
$\tilde{f}(x) = x_n$ and $h(x)=\tau_n$ for $x\in A_n$.

Let $(X_n)_{n\ge 0}$ be the Markov chain obtained by applying to
$\mathcal{S}$, $\tilde{f}$ and $h$ the main construction
introduced in Section \ref{section_counterexample}. By property
(viii) and the condition $\tau_0 = 1$, the chain is Harris ergodic. By
property (v) we have
\begin{align*}
\E_\nu \psi(\tau+1) = R\int_{\mathcal{S}}
\psi(h(x))h(x)^{-1}\alpha(dx) = R\sum_{n\ge
0}\frac{\psi(\tau_n)}{\tau_n}p_n \le 2RC
\end{align*}
by the definition of $p_n$. Thus the chain $(X_n)$ satisfies
$\|\tau\|_{\nu,\psi} < \infty$.

By property (vi) we get
\begin{displaymath}
\E_\pi \rho(f) = \int_{\mathcal{S}} \rho(\tilde{f}(x))\alpha(dx) =
\sum_{n\ge 0} \rho(x_n)p_n \le 2C.
\end{displaymath}

On the other hand for any $\theta > 0$ we have by property (v),
the construction of functions $\tilde{f},g$ and (\ref{weak_opt
nu_tau}),
\begin{align*}
\E_\nu \varphi(\theta |S(f)|) &= R\int_\mathcal{S}
\varphi(\theta|\tilde{f}(x)|h(x))h(x)^{-1}\alpha(dx)\\
&\ge R\sum_{n\ge 1}
\frac{\varphi(2^{n-1}\theta x_n\tau_n2^{1-n})}{\tau_n}p_n
\ge R \sum_{2^{n-1}\theta \ge
1}2^{n-1}\theta\frac{\varphi(x_n\tau_n2^{1-n})}{\tau_n}p_n\\
&\ge R\sum_{2^{n-1}\theta \ge 1} 2^{n-1}\theta\Big(
\rho(x_n)+\frac{\psi(\tau_n)}{\tau_n}\Big)p_n = \infty,
\end{align*}
which shows that $\|S(f)\|_{\nu,\varphi} = \infty$ and proves the
proposition.
\end{proof}

\begin{proof}[Proof of Theorem \ref{stron_opt_nu}]

Let $\mathcal{S}$ be a Polish space, $\alpha$ a probability
measure on $\mathcal{S}$ and $\tilde{f}\colon \mathcal{S} \to \R$
a function whose law under $\alpha$ is the same as the law of $Y$.

We will consider in detail only the case when $\lim_{x\to \infty} \varphi(x)/x= \infty$. It is easy to see using formula (\ref{mean}) and the construction below that in the case $\varphi \simeq id$ the theorem also holds (note that in this case also $\rho \simeq id$).

By the convexity assumption and Lemma \ref{asymptotic_convexity} from the Appendix, we obtain that $\eta = (\psi^\ast)^{-1}\circ \varphi^\ast$ is equivalent to a Young function. Thus by Proposition \ref{legendre_prop} and
Lemma \ref{lemma_inverses} from the Appendix we get
\begin{align}\label{duals_equivalence}
\rho^\ast(\cdot) \simeq (\psi^\ast)^{-1}\circ \varphi^\ast(\cdot)
\simeq \frac{\varphi^\ast(\cdot)}{\psi^{-1}\circ \varphi^\ast
(\cdot)}.
\end{align}
By Lemma \ref{general_duality} in the Appendix (or in the case when $\rho^\ast \simeq id$ by the well known facts about the spaces $L_1$ and $L_\infty$), there exists a function $g\colon
\mathcal{S} \to \R_+$ such that
\begin{align}\label{integrals}
\int_\mathcal{S} \frac{\varphi^\ast(g(x))}{\psi^{-1}( \varphi^\ast
(g(x)))}\alpha(dx) < \infty \;\textrm{and}\;
\int_\mathcal{S}|\tilde{f}(x)|g(x) \alpha(dx) =\infty.
\end{align}

Define the function $h\colon\mathcal{S} \to \N\setminus\{0\}$ by
$h(x) = \lfloor \psi^{-1}(\varphi^\ast(g(x)))\rfloor + 1$.

Let now $\mathcal{X}$, $(X_n)$ and $f$ be the Polish space, Markov
chain and function obtained from $\mathcal{S},\alpha,\tilde{f},h$
according to the main construction of Section
\ref{section_counterexample}. Note that we can assume that
$\alpha(\{x\colon h(x) = 1\}) > 0$ and thus by property (viii)
this chain is Harris ergodic.

Note that by the definition of $h$, if $h(x) \ge 2$ then $h(x) \le
2\psi^{-1}(\varphi^\ast(g(x)))$. Thus by property (v) and
(\ref{integrals}) we get
\begin{align*}
\E_\nu \psi((\tau+1)/2) &=
R\int_{\mathcal{S}}\psi(h(x)/2)h(x)^{-1}\alpha(dx) \\
&\le R\psi(1/2) +
R\int_{\mathcal{S}}\frac{\varphi^\ast(g(x))}{\psi^{-1}(\varphi^\ast(g(x)))}\alpha(dx)
< \infty,
\end{align*}
which implies that $\|\tau\|_{\nu,\psi} < \infty$. Recall now the
definition of $\nu$ given in (\ref{nu_definition}). For all $a >
0$ we have
\begin{displaymath}
\E_\nu \varphi(S(f)/a) =
\int_\mathcal{S}\varphi(\tilde{f}(x)h(x)/a)\nu(dx),
\end{displaymath}
which implies that $\|S(f)\|_{\nu,\varphi} < \infty$ iff  $
\|\tilde{f}h\|_{\nu,\varphi} < \infty$ (note that on the left hand
side $\nu$ is treated as a measure on $\mathcal{X}$ and on the
right hand side as a measure on $\mathcal{S}$).

Note however that by (\ref{integrals}) we have
\begin{align*}
\int_\mathcal{S} \varphi^\ast(g(x))\nu(dx) &=
R\int_\mathcal{S}\frac{\varphi^\ast(g(x))}{h(x)}\alpha(dx)\le
R\int_\mathcal{S}\frac{\varphi^\ast(g(x))}{\psi^{-1}(\varphi^\ast(g(x)))}\alpha(dx)
< \infty,
\end{align*}
which gives $\|g\|_{\nu,\varphi^\ast} < \infty$, but
\begin{align*}
\int_\mathcal{S} \tilde{f}(x) h(x) g(x) \nu(dx) =
R\int_\mathcal{S} \tilde{f}(x)g(x)\alpha(dx) = \infty.
\end{align*}
This shows that $\|\tilde{f}h\|_{\nu,\varphi} = \infty$ and ends
the proof.
\end{proof}

\begin{proof}[Proof of Proposition \ref{Question_2_opt_nu}]
 Note that for any function
$f$ (not necessarily equivalent to a Young function) we have
$f^{\ast\ast} \le f$. Thus by Proposition \ref{legendre_prop} we
have
\begin{displaymath}
\rho_{\varphi,\psi} \preceq \rho \iff ((\psi^\ast)^{-1}\circ
\varphi^\ast)^\ast \preceq \rho \implies \rho^\ast \preceq
(\psi^\ast)^{-1}\circ \varphi^\ast \iff \psi^\ast\circ\rho^\ast
\preceq \varphi^\ast \iff \varphi \preceq
(\psi^\ast\circ\rho^\ast)^\ast,
\end{displaymath}
which ends the proof by Proposition \ref{weak_opt_nu}.
\end{proof}

\subsection{The stationary case}

For the proofs of results concerning optimality of our estimates
for the chain started from $\pi$, we will also use the general
construction from Section \ref{section_counterexample}. As already
mentioned, in this case the problem turns out to be closely
related to the classical theory of point-wise multiplication of
Orlicz spaces (we refer to \cite{ON} for an overview)

\begin{proof}[Proof of Proposition \ref{weak_opt_pi}]
Assume that the function $\zeta$ does not satisfy the condition
$\zeta_{\varphi,\psi} \preceq \zeta$. Thus there exists a sequence
of numbers $x_n \to \infty$, such that $\zeta(x_n) <
\zeta_{\varphi, \psi}(x_n2^{-n})$, i.e. for some sequence $t_n
>0$, $n=1,2,\ldots$, we have
\begin{displaymath}
\varphi(x_n2^{-n}t_n) \ge \zeta(x_n) + \psi(t_n)/t_n.
\end{displaymath}
Similarly as in the proof of Proposition \ref{weak_opt_nu}, we
show that without loss of generality we can assume that $t_n$ are
positive integers and thus the right hand side
above is bounded from below. Let us additionally define $t_0 = 1$.

Let $p_n = C2^{-n}(\zeta(x_n) + \psi(t_n)/t_n)^{-1}$, where $C$ is
such that $\sum_{n\ge 0}p_n =1$ and consider a probability space
$(\mathcal{S},\sigma)$, where $S = \bigcup_n A_n$ with $A_n$
disjoint and $\alpha(A_n) =p_n$ together with two functions
$\tilde{f}\colon \mathcal{S}\to \R$ and $h\colon \mathcal{S}\to
\R$ such that for $x\in A_n$, we have $\tilde{f}(x) =x_n$, $h(x)=
t_n$.

By applying to $\mathcal{S}$,$\tilde f$ and $h$ the general
construction of Section \ref{section_counterexample}, we get a
Harris ergodic Markov chain and a function $f$, which by
properties (v) and (vi) satisfy
\begin{align*}
\E_\nu \psi(\tau+1) &= R\sum_{n\ge 0}\frac{\psi(t_n)}{t_n}p_n \le
2RC,\\
\E_\pi\zeta(f)&= \sum_{n\ge 0}\zeta(x_n)p_n \le C.
\end{align*}
But by property (vii) we get for any $\theta > 0$,
\begin{align*}
\E_\pi \varphi(\theta |S(f)| &\ge \frac{1}{2}\int_{\mathcal{S}}
\varphi(\theta h(x)|\tilde{f}(x)|/2)\alpha(dx)=
\frac{C}{2}\sum_{n\ge 0}\varphi(\theta x_n t_n/2)p_n\\
&\ge \frac{C}{2}\sum_{n\ge
1}\varphi(2^{n-1}\theta x_nt_n2^{-n})p_n \ge
\frac{C}{2} \sum_{2^{n-1}\theta\ge 1}2^{n-1}\theta\varphi(x_nt_n2^{-n})p_n \\
&\ge \frac{C}{2}\sum_{2^{n-1}\theta\ge 1}2^{n-1}\theta(\zeta(x_n) +
\psi(t_n)/t_n)p_n = \infty,
\end{align*}
which ends the proof.
\end{proof}

\begin{proof}[Proof of Theorem \ref{stron_opt_pi}]
Consider first the case $\lim_{x\to \infty}\eta(x)/x = \infty$.

We will show that for some constant $C$ and $x$ large enough
we have
\begin{align}\label{Mal_condition}
\varphi^{-1}(x) \le C\zeta^{-1}(x)\tilde{\psi}^{-1}(x).
\end{align}
We have $\eta^{-1}(x)= \tilde{\psi}^{-1}(\varphi(x))$ and thus by
the assumption on $\eta$ and Lemma \ref{lemma_inverses} from the
Appendix, we get $\varphi^{-1}(x) \le
C(\eta^\ast)^{-1}(\varphi^{-1}(x))\tilde{\psi}^{-1}(x)$ for some
constant $C < \infty$ and $x$ large enough. But by Proposition
\ref{legendre_prop_pi}, $(\eta^\ast)^{-1}(\varphi^{-1}(x)) \le
2\zeta^{-1}(x)$ and thus (\ref{Mal_condition}) follows.

If $\lim_{x\to \infty}\eta(x)/x < \infty$ then (\ref{Mal_condition}) also holds if we interpret $\zeta^{-1}$ as the generalized inverse (note that in this case $L_\zeta =L_\infty$)

Theorem 1 from \cite{MalNa} states that if
$\varphi,\zeta,\tilde{\psi}$ are Young functions such that
(\ref{Mal_condition}) holds for all $x \in [0,\infty)$ and $Y$ is
a random variable such that $\|Y\|_{\zeta} = \infty$, then there
exists a random variable $X$, such that $\|X\|_{\tilde{\psi}} <
\infty$ and $\|XY\|_{\varphi} = \infty$. One can easily see that
the functions $\varphi,\zeta,\tilde{\psi}$ can be modified (for
small values of $x$) to Young functions such that
(\ref{Mal_condition}) holds for all $x \ge 0$. Thus there exists
$X$ satisfying the above condition. Clearly one can assume that
with probability one $X$ is a positive integer and $\p(X=1)>0$.
Consider now a Polish space $(\mathcal{S},\alpha)$ and $\tilde{f},
h\colon\mathcal{S} \to \R$ such that $(\tilde{f},h)$ is
distributed like $(Y,X)$. Let $(X_n)$ be the Markov chain given by
the construction of Section \ref{section_counterexample}. By
property (v) we have
\begin{displaymath}
\E_\nu \psi\Big(\frac{\tau+1}{a}\Big) = R \int_\mathcal{S}
\psi\Big(\frac{h(x)}{a}\Big)h(x)^{-1}\alpha(dx) =
\frac{R}{a}\E\tilde{\psi}\Big(\frac{X}{a}\Big) < \infty
\end{displaymath}
for $a$ large enough, since $\|X\|_{\tilde{\psi}} < \infty$. By
property (vi), the law of $f$ under $\pi$ is equal to the law of
$Y$. Finally, by property (vii), for every $a>0$,
\begin{displaymath}
\E_\pi\varphi\Big(\frac{|S(f)|}{a}\Big) \ge
2^{-1}\E\varphi\Big(\frac{XY}{2a}\Big) = \infty,
\end{displaymath}
which proves that $\|S(f)\|_{\pi,\varphi} = \infty$.
\end{proof}

\begin{proof}[Proof of Proposition \ref{Question_2_opt_pi}]
Let $\eta = \varphi^{-1}\circ\tilde{\psi}$. By Propositions
\ref{weak_opt_pi}, \ref{legendre_prop_pi} and Lemma
\ref{auxiliary_lemma} we have
\begin{displaymath}
\varphi\circ {\eta^\ast} \preceq \zeta \simeq \kappa \circ
\vartheta^\ast,
\end{displaymath}
where $\vartheta = \kappa^{-1}\circ\tilde{\psi}$. Thus
$(\vartheta^\ast)^{-1}\circ\kappa^{-1} \preceq
(\eta^\ast)^{-1}\circ \varphi^{-1}$. Another application of Lemma
\ref{auxiliary_lemma} together with Lemma \ref{lemma_inverses} in
the Appendix yield for some constant $C \in (1,\infty)$ and $x$
large enough,
\begin{align*}
\kappa^{-1}(x)&\le C(\vartheta^\ast)^{-1}(\kappa^{-1}(x))
\tilde{\psi}^{-1}(x) \le C^2(\eta^\ast)^{-1}(\varphi^{-1}(Cx))
\tilde{\psi}^{-1}(Cx)\\
&=  C^2
(\eta^\ast)^{-1}(\varphi^{-1}(Cx))\eta^{-1}(\varphi^{-1}(Cx)) \le
2C^2\varphi^{-1}(Cx)
\end{align*}
which implies that $\varphi \preceq \kappa$.
\end{proof}

\section{Applications\label{Section_applications}}

\subsection{Limit theorems for additive functionals}
It is well known that for a Harris ergodic Markov chain and a
function $f$, the CLT
\begin{align}\label{CLT_eq}
\frac{f(X_0)+\ldots+f(X_{n-1})}{\sqrt{n}} \stackrel{d}{\to}
\mathcal{N}(0,\sigma_f^2)
\end{align}
holds in the stationary case iff it holds for any initial
distribution.

Moreover (see \cite{Chen} and \cite{BLL}) under the assumption
that $\E_\pi f^2 < \infty$, the above CLT holds iff $\E_\nu S(f) =
0$, $\E_\nu(S(f))^2 < \infty$ and the asymptotic variance is given
by $\sigma_f^2 = \delta\pi(C)m^{-1}(\E s_{1}(f)^2 +2\E
s_1(f)s_2(f))$. If the chain has an atom, this
equivalence holds without the assumption $\E_\pi f^2 < \infty$.

It is also known (see \cite{Chen}) that the condition $\E_\pi f =
0$, $\E_\nu S(|f|)^2 < \infty$ implies the law of the iterated
logarithm
\begin{align}\label{LIL_eq}
-\sigma_f = \liminf_{n\to \infty}
\frac{\sum_{i=0}^{n-1}f(X_i)}{\sqrt{n\log\log n}} \le
\limsup_{n\to \infty}
\frac{\sum_{i=0}^{n-1}f(X_i)}{\sqrt{n\log\log n}} = \sigma_f \;
a.s.
\end{align}
Moreover for chains with an atom $\limsup_{n\to \infty}
\frac{|\sum_{i=0}^{n-1}f(X_i)|}{\sqrt{n\log\log n}}< \infty
\;a.s.$ implies the CLT (see \cite{Chen}, Theorem 2.2. and Remark
2.3).

Our results from section \ref{section_nu_thm} can be thus applied
to give optimal conditions for CLT and LIL in terms of ergodicity
of the chain (expressed by Orlicz integrability of the
regeneration time) and integrability of $f$ wrt the stationary
measure.

The following Theorem is an immediate consequence of Theorems
\ref{thm_nu}, \ref{stron_opt_nu} and Proposition
\ref{weak_opt_nu}.

\begin{thm}\label{CLT_LIL_thm} Consider a Harris ergodic Markov chain $(X_n)$ on a
Polish space $\mathcal{X}$ and a function $f\colon \mathcal{X} \to
\R$, $\E_\pi f = 0$. Let $\psi$ be a Young function such that
$\lim_{x\to 0}\psi(x)/x = 0$ and assume that $\|\tau\|_{\nu,\psi}
< \infty$. Let finally $\rho(x) = \tilde{\psi}^\ast(x^2)$, where
$\tilde{\psi}(x) = \psi(x)/x$. If $\|f\|_{\pi,\rho} <\infty$ then
the CLT (\ref{CLT_eq}) and LIL (\ref{LIL_eq}) hold.

Moreover every Young function $\tilde{\rho}$ such that
$\|f\|_{\pi,\tilde{\rho}}$ implies CLT (or LIL) for all Harris
ergodic Markov chains with $\|\tau\|_{\nu,\psi} < \infty$
satisfies $\rho\preceq \tilde{\rho}$.

If the function $x\mapsto \sqrt{\psi(x)}$ is equivalent to a Young
function then for every random variable $Y$ with $\|Y\|_{\rho} =
\infty$ one can construct a stationary Harris ergodic Markov chain
$(X_n)$ and a function $f$ such that $f(X_n)$ has the same law as
$Y$, $\|\tau\|_{\nu,\psi}< \infty$ and both (\ref{CLT_eq}) and
(\ref{LIL_eq}) fail.
\end{thm}

\paragraph{Remark} As noted in \cite{Jo} in the case of geometric ergodicity, i.e. when
$\psi(x) = \exp(x) - 1$, the CLT part of the above theorem can be
obtained from results in \cite{DMR}, giving very general and
optimal conditions for CLT under $\alpha$-mixing. The
integrability condition for $f$ is in this case $\E_\pi
f^2\log_+(|f|) < \infty$. The sufficiency for the LIL part can be
similarly deduced from \cite{RioLil}.

The equivalence of the exponential decay of mixing coefficients
with ergodicity of Markov chains (measured in terms of $\psi$)
follows from \cite{NT1,NT2}. Optimality of the condition follows
from examples given in \cite{Bra2}. Examples of geometrically
ergodic Markov chains and a function $f$ such that $\E f^2<\infty$
and the CLT fails have been also constructed in \cite{Hag,Bra1}.
Let us point out that if the Markov chain is reversible and
geometrically ergodic, then $\|f\|_{\pi,2} < \infty$ implies the
CLT and thus also $\E_\nu S(f)^2 < \infty$. Thus under this
additional assumptions our formulas for $\psi(x) = \exp(x)-1$ and
$\rho(x) = x^2$ are no longer optimal (our example from Section
\ref{section_counterexample} is obviously non-reversible). It
would be of interest to derive counterparts of theorems from
Section \ref{Section_main} under the assumption of reversibility.

It is possible that in a more general case Theorem
\ref{CLT_LIL_thm} can also be recovered from the above results, by
proper characterizations of ergodicity in terms of mixing and
characterizations of Orlicz spaces in terms of some weighted
inequalities involving the tail of the function. However we have
not attempted to do this in full generality (we have only verified
that such an approach works in the case of $\psi(x) =x^p$).

Let us also remark that to our best knowledge, so far there has
been no 'regeneration' proof of Theorem \ref{CLT_LIL_thm} even in
the case of geometric ergodicity.

\paragraph{Berry-Esseen type theorems}
Similarly we can use a result by Bolthausen \cite{Bo1,Bo2} to
derive Berry-Esseen type bounds for additive functionals of
stationary chains. More specifically Lemma 2 in \cite{Bo2}, together with Theorem \ref{thm_nu} give

\begin{thm} Let $(X_n)$ be a stationary strongly aperiodic Harris ergodic Markov chain on $\mathcal{X}$, such that $\|\tau\|_{\nu,\psi}<\infty$, where $\psi$ is a Young function
satisfying $(x\mapsto x^3) \preceq \psi$ and $\lim_{x\to0}\psi(x)/x
= 0$. Let $\rho = \Psi^\ast(x^3)$, where $\Psi(x)
= \psi(\sqrt{x})/\sqrt{x}$. Then for every $f\colon \mathcal{X}\to \R$ such that
$\|f\|_{\pi,\rho} < \infty$ and $\sigma_f^2 := \E(S(f))^2 > 0$ we
have
\begin{displaymath}
\Big|\p\Big( \frac{\sum_{i=0}^{n-1} f(X_i)-\E_\pi
f}{\sigma_f\sqrt{n}}\Big) - \Phi(x)\Big| = \mathcal{O}(n^{-1/2}),
\end{displaymath}
where $\Phi(x) = (2\pi)^{-1/2}\int_{-\infty}^x\exp(-y^2/2)dy$.

\end{thm}

\subsection{Tail estimates}

The last application we develop concerns tail inequalities for
additive functionals. The approach we take is by now fairly
standard (see e.g. \cite{Cle,DGM,BerCle,AdMarkovTail,AdBed,LMN_MCMC,LMN_MCQMC}) and relies
on splitting the additive functional into a sum of independent (or
one-dependent blocks) and using inequalities for sums of
independent random variables. Our results on Orlicz integrability
imply inequalities for the chain started from the small measure
(an atom) or from the stationary distribution. The former case may
have potential applications in MCMC algorithms in situations when
small measure is known explicitly and one is able to sample from
it.

In what follows we denote $\psi_\alpha = \exp(x^\alpha) -
1$.

\begin{thm}\label{conc_exp} Let $(X_n)_{n\ge 0}$ be a Harris ergodic Markov chain on $\mathcal{X}$.
Assume that
$\|\tau\|_{\nu,\psi_\alpha} < \infty$ for some $\alpha \in(0,1)$.
Let $f\colon\mathcal{X}\to \R$ be a measurable function, such that
$\E_\pi f = 0$. If $\|f\|_{\pi,\psi_\beta} < \infty$ for some $\beta
>0$, then for all $t\ge 0$,

\begin{align*}
&\p_\nu(|f(X_0)+\ldots+f(X_{n-1})| \ge t) \nonumber \\
&\le K\exp\Big(-\frac{t^2}{Kn\delta \pi(C)\E_\nu S(f)^2}\Big) +
K\exp\Big(-\frac{t}{K\|f\|_{\pi,\psi_\beta}\|\tau+1\|_{\nu,\psi_\alpha}^3}\Big)\\
&+
K\exp\Big(\frac{t^{\gamma}}{K(\|f\|_{\pi,\psi_\beta}\|\tau+1\|_{\nu,\psi_\alpha})^\gamma\log
n}\Big)
\end{align*}
and
\begin{align*}
&\p_\pi(|f(X_0)+\ldots+f(X_{n-1})| \ge t) \nonumber \\
&\le K\exp\Big(-\frac{t^2}{Kn\delta \pi(C)\E_\nu S(f)^2}\Big) +
K\exp\Big(-\frac{t}{K\|f\|_{\pi,\psi_\beta}\|\tau+1\|_{\nu,\psi_\alpha}^3}\Big)\\
&+
K\exp\Big(\frac{t^\gamma}{K(\|f\|_{\pi,\psi_\beta}\|\tau+1\|_{\nu,\psi_\alpha})^\gamma
\log(\|\tau+1\|_{\nu,\psi_\alpha})}\Big) +
K\exp\Big(\frac{t^{\gamma}}{K(\|f\|_{\pi,\psi_\beta}\|\tau+1\|_{\nu,\psi_\alpha})^\gamma\log
n}\Big),
\end{align*}
where $\gamma = \frac{\alpha\beta}{\alpha+\beta}$ and $K$ depends
only on $\alpha,\beta$ and $m$ in the formula (\ref{small_set}).
\end{thm}

\paragraph{Remarks}
\begin{enumerate}
\item The proof of the above theorem is similar to those presented
in \cite{AdMarkovTail,AdBed}, therefore we will present only a
sketch.

\item When $m=1$, $\delta\pi(C)\E_\nu S(f)^2$ is the variance of
the limiting Gaussian distribution for the additive functional.

\item If one does not insist on having the limiting variance in
the case $m=1$ as the subgaussian coefficient and instead
replaces it by $\E_\nu S(f)^2$, one can get rid of the second
summand on the right hand sides of the estimates (i.e. the
summand containing $\|\tau+1\|^3$).

\item One can also obtain similar results for suprema of empirical
processes of a Markov chain (or equivalently for additive
functionals with values in a Banach space). The difference is that
one obtains then bounds on deviation above expectation and not
from zero. A proof is almost the same, it simply requires a
suitable generalization of an inequality for real valued summands,
relying on the celebrated Talagrand's inequality and an additional
argument to take care of the expectation. Since our goal is rather
to illustrate the consequences of results from Section
\ref{Section_main}, than to provide the most general inequalities,
we do not state the details and refer the reader to
\cite{AdMarkovTail,AdBed} for the special case of geometrically
ergodic Markov chains. For the same reason we will not try to
evaluate constants in the inequalities.

\item In a similar way one can obtain tail estimates in the polynomial case (i.e. when the regeneration time and/or the function $f$ are only polynomially integrable). One just needs to use the other examples that have been discussed in Section \ref{Section_main}. The estimate of the bounded part (after truncation) comes again from Bernstein's inequality, whereas the unbounded part can be handled with the Hoffman-Joergensen inequality (or its easy modifications for functions of the form $x\mapsto x^p/(\log^\beta x)$), just as e.g. in \cite{EinmahlLi}.
\end{enumerate}
\begin{proof}[Proof of Theorem \ref{conc_exp}]
Below we will several times use known bounds for sums of
independent random variables in the one-dependent case. Clearly it
may be done at the cost of worsening the constants by splitting
the sum into sums of odd and even terms, therefore we will just
write the final result without further comments. In the proof we
will use the letter $K$ to denote constants depending on
$\alpha,\beta$. Their values may change from one occurrence to
another.

Setting $N = \inf\{i\colon m\tau(i) + m -1 \ge n-1\}$ we may write
\begin{align*}
|f(X_0) + \ldots + f(X_{n-1})| &= \sum_{i=0}^{(m\tau(0)+m-1)}
|f(X_i)| +
|\sum_{i=1}^{N} s_i(f)| + \sum_{i=n}^{m\tau(N)+m-1} |f(X_i)|\\
&=: I + II + III,
\end{align*}
where each of the sums on the right hand side may be interpreted
as empty.

The first and last terms can be taken care of by Chebyshev's
inequalities corresponding to proper Orlicz norms, using estimates
of Corollaries \ref{question_2_cor_nu} and \ref{question_2_cor_pi}
(note that $\p(III \ge t) \le \p(I\ge t) + n\p(|s_1(f)|\ge t)$).

We will consider only the case of the chain started from $\nu$.
The stationary case is similar, simply to bound $I$ we use the
estimates of Orlicz norms for the chain started from $\pi$ given
in Theorem \ref{thm_pi} (together with the remark following it to
get better dependence on $\|\tau+1\|_{\nu,\psi_\alpha}$).

 By
Corollary \ref{question_2_cor_nu} and examples provided in Section \ref{Section_main}, $\|s_i(f)\|_\varphi =
\|S(f)\|_{\nu,{\psi_\gamma}} <
K\|\tau+1\|_{\nu,\psi_\alpha}\|f\|_{\pi,\psi_\beta}$. Thus
\begin{align}\label{I_i_III}
\p(I \ge t) + \p(III \ge t) \le
2n\exp\Big(\Big(-\frac{t}{K\|\tau+1\|_{\nu,\psi_\alpha}\|f\|_{\pi,\psi_\beta}}\Big)^\gamma\Big).
\end{align}

The second term can be split into $II_1 + II_2$, where
\begin{align*}
II_1& =|\sum_{i=1}^N (s_i(f)\ind{|s_i(f)|\le a} - \E
s_i(f)\ind{|s_i(f)|\le a})|, \\
II_2 & = |\sum_{i=1}^N (s_i(f)\ind{|s_i(f)|> a} - \E
s_i(f)\ind{|s_i(f)|> a})|.
\end{align*}

Setting $a = K_{\alpha,\beta} \max_{i\le n}
\|s_i(f)\|_{\psi_\gamma}\log^{1/\gamma} n \le K_{\alpha,\beta} m
\|f\|_{\pi,\psi_\beta}\|\tau+1\|_{\psi_\alpha}\log^{1/\gamma} n$,
we can proceed as in \cite{AdBed} to get
\begin{align}\label{II_2}
\p(II_2 \ge t) \le 2\exp\Big(\Big(-\frac{t}{K_{\alpha,\beta}
a}\Big)^\gamma\Big).
\end{align}
 It remains to bound the term $I_1$. Introduce the variables
$T_i = \tau(i) - \tau(i-1)$, $i\ge 1$ and note that $\E T_i =
\delta^{-1}\pi(C)^{-1}$. For $4nm^{-1}\pi(C)\delta \ge 2$, we have
\begin{align*}
\p(N \ge 4nm^{-1}\pi(C)\delta) &\le \p\Big(
\sum_{i=1}^{\lfloor 4nm^{-1}\pi(C)\delta\rfloor} T_i \le n/m\Big)\\
&= \p\Big( \sum_{i=1}^{\lfloor 4nm^{-1}\pi(C)\delta\rfloor} (T_i -
\E T_i) \le n/m - 2nm^{-1}\pi(C)\delta \E T_i\Big)\\
& = \p\Big( \sum_{i=1}^{\lfloor 4nm^{-1}\pi(C)\delta\rfloor} (T_i
- \E T_i) \le - n/m\Big)\\
& \le k(n/m),
\end{align*}
where $k(t) =
K_\alpha\exp(-K_\alpha^{-1}\min(t^2m/n\|\tau+1\|_{\nu,\psi_\alpha}^2,(t/\|\tau+1\|_{\nu,\psi_\alpha})^\alpha))$.
The bound follows for $\alpha = 1$ from
Bernstein's $\psi_1$ inequality and for $\alpha < 1$ from results
in \cite{HOMS} (as shown in \cite{ALPT2}).

Note that if $4nm^{-1}\pi(C)\delta < 2$, then
\begin{align*}
k(n/m) &\ge \exp(-K_\alpha n m^{-1} \|\tau+1\|_{\nu,\psi_\alpha}^{-2}) \ge
\exp(-K_\alpha n m^{-1} (\E \tau+1)^{-2}) \\
&= \exp(-K_\alpha nm^{-1}\delta\pi(C)) \ge \exp(-K_\alpha/2).
\end{align*}
Thus the above tail estimate for $N$ remains true (after adjusting the constant $K_\alpha$).

Thus by Bernstein's bounds on suprema of partial sums of  a
sequence of independent random variables we get
\begin{align}\label{II_11}
\p(II_1 \ge t) &\le \p(II_1 \ge t \; \&\; N \le
4nm^{-1}\pi(C)\delta) + k(n/m) \nonumber \\
&\le
K\exp\Big(-\frac{1}{K}\min\Big(\frac{t^2}{nm^{-1}\pi(C)\delta\E_\nu
S(f)^2},\frac{t}{\|f\|_{\pi,\psi_\beta}\|\tau\|_{\psi_\alpha}}\Big)\Big)
+ k(n/m),
\end{align}

On the other hand by $N \le n/m$, the same
inequalities we used to derive the function $k$ and Levy type
inequalities (like in \cite{AdMarkovTail}), we obtain
\begin{align}\label{II_12}
\p(II \ge t) \le h(t),
\end{align}
where $h(t) =
K_\gamma\exp(-K_\gamma^{-1}
\min(t^2m/n(\|f\|_{\pi,\psi_\beta}\|\tau+1\|_{\nu,\psi_\alpha})^2,(t/\|f\|_{\pi,\psi_\beta}\|\tau+1\|_{\nu,\psi_\alpha})^\gamma))$

Now if $k(n/m) \ge
K\exp(-(t/K\|f\|_{\pi,\psi_\beta}\|\tau+1\|_{\nu,\psi_\alpha})^\alpha)$,
then
\begin{displaymath}
K\exp\Big(-\Big(\frac{t}{K\|f\|_{\pi,\psi_\beta}\|\tau+1\|_{\psi_\alpha}}\Big)^\alpha\Big)
\le k(n/m) \le
K\exp\Big(-\Big(\frac{n}{Km\|\tau+1\|_{\nu,\psi_\alpha}^2}\Big)^\alpha\Big)
\end{displaymath}
and so
\begin{align*}
\frac{t^2m}{n(\|f\|_{\pi,\psi_\beta}\|\tau+1\|_{\nu,\psi_\alpha})^2}
\ge
\frac{t}{\|\tau+1\|_{\nu,\psi_\alpha}^3\|f\|_{\pi,\psi_\beta}},
\end{align*}
which ends the proof by (\ref{II_12}).
\end{proof}

\section*{Appendix. Some generalities on Orlicz Young functions and
Orlicz spaces}

All the lemmas presented below are standard facts from the theory
of Orlicz spaces, we present them here for the reader's
convenience.

\begin{lemma}\label{general_duality}
If $\varphi$ is a Young function then $X \in L_\varphi$ if and
only if $\E |XY| < \infty$ for all $Y$ such that $\E
\varphi^\ast(Y) \le 1$. Moreover the norm $\|X\| = \sup\{\E XY
\colon \E\varphi^\ast(Y) \le 1\}$ is equivalent to
$\|X\|_{\varphi}$.
\end{lemma}

The next lemma is a modification of Lemma  5.4. in \cite{MS}. In the original formulation it concerns the notion of equivalence of functions (and not asymptotic equivalence relevant in our probabilistic setting). One can however easily see that the proof from \cite{MS} yields the version stated below.

\begin{lemma}\label{asymptotic_convexity_MS} Consider two increasing continuous functions $F,G \colon
[0,\infty) \to [0,\infty)$ with $F(0) = G(0) = 0$, $F(\infty)=
G(\infty) = \infty$. The following conditions are equivalent
\begin{itemize}
\item[(i)] $F \circ G^{-1}$ is equivalent to a Young function.

\item[(ii)] There exist positive constants $C,x_0$ such that
\begin{displaymath}
F\circ G^{-1}(sx) \ge C^{-1}sF\circ G^{-1}(x)
\end{displaymath}
for all $s \ge 1$ and $x \ge x_0$.

\item[(iii)] There exist positive constants $C,x_0$ such that
\begin{displaymath}
\frac{F(sx)}{F(x)} \ge C^{-1}\frac{G(sx)}{G(x)}
\end{displaymath}
for all $s\ge 1$, $x\ge x_0$.
\end{itemize}
\end{lemma}

\begin{lemma}\label{lemma_inverses}
For any Young function $\psi$ such that $\lim_{x\to \infty} \psi(x)/x = \infty$ and any $x \ge 0$,
\begin{displaymath}
x \le (\psi^\ast)^{-1}(x)\psi^{-1}(x) \le 2x.
\end{displaymath}
Moreover the right hand side inequality holds for any strictly
increasing function $\psi\colon [0,\infty) \to [0,\infty)$ with
$\psi(0) = 0$, $\psi(\infty) = \infty$, $\lim_{x\to \infty} \varphi(x)/x = \infty$.
\end{lemma}

\begin{lemma}\label{asymptotic_convexity}
Let $\varphi$ and $\psi$ be two Young functions. Assume that $\lim_{x\to \infty} \varphi(x)/x = \infty$. If $\varphi^{-1} \circ \psi$ is equivalent to a Young function, then so is $(\psi^\ast)^{-1}\circ \varphi^\ast$.
\end{lemma}

\begin{proof} It is easy to see that under the assumptions of the lemma we also have $\lim_{x\to \infty}\psi(x)/x = \infty$ and thus $\varphi^\ast(x)$, $\psi^\ast(x)$  are finite for all $x \ge 0$. Applying Lemma \ref{asymptotic_convexity_MS} with $F = \varphi^{-1}$, $G = \psi^{-1}$, we get that
\begin{displaymath}
\frac{\varphi^{-1}(sx)}{\psi^{-1}(sx)} \ge C^{-1}\frac{\varphi^{-1}(x)}{\psi^{-1}(x)}
\end{displaymath}
for some $C > 0$, all $s \ge 1$ and $x$ enough. By Lemma \ref{lemma_inverses} we obtain for $x$ large enough,
\begin{displaymath}
\frac{(\psi^\ast)^{-1}(sx)}{(\varphi^\ast)^{-1}(sx)} \ge (4C)^{-1}\frac{(\psi^\ast)^{-1}(x)}{(\varphi^\ast)^{-1}(x)},
\end{displaymath}
which by another application of Lemma \ref{asymptotic_convexity_MS} ends the proof.

\end{proof}

\bibliographystyle{abbrv}

\begin{thebibliography}{10}

\bibitem{AdMarkovTail}
R.~Adamczak.
\newblock A tail inequality for suprema of unbounded empirical processes with
  applications to {M}arkov chains.
\newblock {\em Electron. J. Probab.}, 13:no. 34, 1000--1034, 2008.

\bibitem{AdBed}
R.~Adamczak and W.~Bednorz.
\newblock Exponential inequalities for additive functionals of {M}arkov chains. Preprint.

\bibitem{ALPT2}
R.~Adamczak, A.~E. Litvak, A.~Pajor, and N.~Tomczak-Jaegermann.
\newblock Restricted isometry property of matrices with independent columns and
  neighborly polytopes by random sampling.
\newblock {\em to appear in Constr. Approx.}

\bibitem{ANsplit}
K.~B. Athreya and P.~Ney.
\newblock A new approach to the limit theory of recurrent {M}arkov chains.
\newblock {\em Trans. Amer. Math. Soc.}, 245:493--501, 1978.

\bibitem{Bax}
P.~H. Baxendale.
\newblock Renewal theory and computable convergence rates for geometrically
  ergodic {M}arkov chains.
\newblock {\em Ann. Appl. Probab.}, 15(1B):700--738, 2005.

\bibitem{BLL}
W.~Bednorz, K.~{\L}atuszy{\'n}ski, and R.~Lata{\l}a.
\newblock A regeneration proof of the central limit theorem for uniformly
  ergodic {M}arkov chains.
\newblock {\em Electron. Commun. Probab.}, 13:85--98, 2008.

\bibitem{BerCle}
P.~Bertail and S.~Cl{\'e}men{\c{c}}on.
\newblock Sharp bounds for the tails of functionals of {M}arkov chains.
\newblock {\em Teor. Veroyatn. Primen.}, 54(3):609--619, 2009.

\bibitem{Bo1}
E.~Bolthausen.
\newblock The {B}erry-{E}sseen theorem for functionals of discrete {M}arkov
  chains.
\newblock {\em Z. Wahrsch. Verw. Gebiete}, 54(1):59--73, 1980.

\bibitem{Bo2}
E.~Bolthausen.
\newblock The {B}erry-{E}sse\'en theorem for strongly mixing {H}arris recurrent
  {M}arkov chains.
\newblock {\em Z. Wahrsch. Verw. Gebiete}, 60(3):283--289, 1982.

\bibitem{Bra2}
R.~C. Bradley.
\newblock On quantiles and the central limit question for strongly mixing
  sequences.
\newblock {\em J. Theoret. Probab.}, 10(2):507--555, 1997.
\newblock Dedicated to Murray Rosenblatt.

\bibitem{Bra1}
R.~C. Bradley, Jr.
\newblock Information regularity and the central limit question.
\newblock {\em Rocky Mountain J. Math.}, 13(1):77--97, 1983.

\bibitem{Chen}
X.~Chen.
\newblock Limit theorems for functionals of ergodic {M}arkov chains with
  general state space.
\newblock {\em Mem. Amer. Math. Soc.}, 139(664):xiv+203, 1999.

\bibitem{Cle}
S.~J.~M. Cl{\'e}men{\c{c}}on.
\newblock Moment and probability inequalities for sums of bounded additive
  functionals of regular {M}arkov chains via the {N}ummelin splitting
  technique.
\newblock {\em Statist. Probab. Lett.}, 55(3):227--238, 2001.

\bibitem{DFMS}
R.~Douc, G.~Fort, E.~Moulines, and P.~Soulier.
\newblock Practical drift conditions for subgeometric rates of convergence.
\newblock {\em Ann. Appl. Probab.}, 14(3):1353--1377, 2004.

\bibitem{DGM}
R.~Douc, A.~Guillin, and E.~Moulines.
\newblock Bounds on regeneration times and limit theorems for subgeometric
  {M}arkov chains.
\newblock {\em Ann. Inst. Henri Poincar\'e Probab. Stat.}, 44(2):239--257,
  2008.

\bibitem{DMR}
P.~Doukhan, P.~Massart, and E.~Rio.
\newblock The functional central limit theorem for strongly mixing processes.
\newblock {\em Ann. Inst. H. Poincar\'e Probab. Statist.}, 30(1):63--82, 1994.

\bibitem{EinmahlLi}
U.~Einmahl and D.~Li.
\newblock Characterization of {LIL} behavior in {B}anach space.
\newblock {\em Trans. Amer. Math. Soc.}, 360(12):6677--6693, 2008.

\bibitem{Hag}
O.~H{\"a}ggstr{\"o}m.
\newblock On the central limit theorem for geometrically ergodic {M}arkov
  chains.
\newblock {\em Probab. Theory Related Fields}, 132(1):74--82, 2005.

\bibitem{HOMS}
P.~Hitczenko, S.~J. Montgomery-Smith, and K.~Oleszkiewicz.
\newblock Moment inequalities for sums of certain independent symmetric random
  variables.
\newblock {\em Studia Math.}, 123(1):15--42, 1997.

\bibitem{Jo}
G.~L. Jones.
\newblock On the {M}arkov chain central limit theorem.
\newblock {\em Probab. Surv.}, 1:299--320 (electronic), 2004.

\bibitem{Kr}
M.~A. Krasnoselski{\u\i} and J.~B. Ruticki{\u\i}.
\newblock {\em Convex functions and {O}rlicz spaces}.
\newblock Translated from the first Russian edition by Leo F. Boron. P.
  Noordhoff Ltd., Groningen, 1961.

\bibitem{LMN_MCMC}
K.~{\L}atuszynski, B.~Miasojedow, and W.~Niemiro.
\newblock Nonasymptotic bounds on the estimation error of {M}{C}{M}{C}
  algorithms.
\newblock {\em {A}vailable at http://arxiv.org/abs/1106.4739}.

\bibitem{LMN_MCQMC}
K.~{\L}atuszynski, B.~Miasojedow, and W.~Niemiro.
\newblock Nonasymptotic bounds on the mean square error for {M}{C}{M}{C}
  estimates via renewal techniques.
\newblock {\em To appear in MCQMC 2010 Conference Proceeding. Available at
  http://arxiv.org/abs/1101.5837}.

\bibitem{MalBook}
L.~Maligranda.
\newblock {\em Orlicz spaces and interpolation}, volume~5 of {\em Semin\'arios
  de Matem\'atica [Seminars in Mathematics]}.
\newblock Universidade Estadual de Campinas, Departamento de Matem\'atica,
  Campinas, 1989.

\bibitem{MalNa}
L.~Maligranda and E.~Nakai.
\newblock Pointwise multipliers of {O}rlicz spaces.
\newblock {\em Arch. Math. (Basel)}, 95(3):251--256, 2010.

\bibitem{MT}
S.~Meyn and R.~L. Tweedie.
\newblock {\em Markov chains and stochastic stability}.
\newblock Cambridge University Press, Cambridge, second edition, 2009.
\newblock With a prologue by Peter W. Glynn.

\bibitem{MS}
S.~J. Montgomery-Smith.
\newblock Comparison of {O}rlicz-{L}orentz spaces.
\newblock {\em Studia Math.}, 103(2):161--189, 1992.

\bibitem{NummSplit}
E.~Nummelin.
\newblock A splitting technique for {H}arris recurrent {M}arkov chains.
\newblock {\em Z. Wahrsch. Verw. Gebiete}, 43(4):309--318, 1978.

\bibitem{Numm}
E.~Nummelin.
\newblock {\em General irreducible {M}arkov chains and nonnegative operators},
  volume~83 of {\em Cambridge Tracts in Mathematics}.
\newblock Cambridge University Press, Cambridge, 1984.

\bibitem{NT1}
E.~Nummelin and P.~Tuominen.
\newblock Geometric ergodicity of {H}arris recurrent {M}arkov chains with
  applications to renewal theory.
\newblock {\em Stochastic Process. Appl.}, 12(2):187--202, 1982.

\bibitem{NT2}
E.~Nummelin and P.~Tuominen.
\newblock The rate of convergence in {O}rey's theorem for {H}arris recurrent
  {M}arkov chains with applications to renewal theory.
\newblock {\em Stochastic Process. Appl.}, 15(3):295--311, 1983.

\bibitem{ON}
R.~O'Neil.
\newblock Fractional integration in {O}rlicz spaces. {I}.
\newblock {\em Trans. Amer. Math. Soc.}, 115:300--328, 1965.

\bibitem{Pit1}
J.~W. Pitman.
\newblock An identity for stopping times of a {M}arkov process.
\newblock In {\em Studies in probability and statistics (papers in honour of
  {E}dwin {J}. {G}. {P}itman)}, pages 41--57. North-Holland, Amsterdam, 1976.

\bibitem{Pit2}
J.~W. Pitman.
\newblock Occupation measures for {M}arkov chains.
\newblock {\em Advances in Appl. Probability}, 9(1):69--86, 1977.

\bibitem{RaoRen}
M.~M. Rao and Z.~D. Ren.
\newblock {\em Theory of {O}rlicz spaces}, volume 146 of {\em Monographs and
  Textbooks in Pure and Applied Mathematics}.
\newblock Marcel Dekker Inc., New York, 1991.

\bibitem{RioLil}
E.~Rio.
\newblock The functional law of the iterated logarithm for stationary strongly
  mixing sequences.
\newblock {\em Ann. Probab.}, 23(3):1188--1203, 1995.

\bibitem{RobRos}
G.~O. Roberts and J.~S. Rosenthal.
\newblock General state space {M}arkov chains and {MCMC} algorithms.
\newblock {\em Probab. Surv.}, 1:20--71 (electronic), 2004.

\end{thebibliography}

\end{document}